\documentclass[12pt]{amsart}
\usepackage{amssymb,amsfonts,latexsym,amscd}
\usepackage{extpfeil}
\usepackage{mathrsfs}

\usepackage[all,cmtip,2cell]{xy}
\UseAllTwocells
\usepackage{verbatim}
\usepackage{hyperref}

\newtheorem{theorem}{Theorem}[section]
\newtheorem{lemma}[theorem]{Lemma}
\newtheorem{proposition}[theorem]{Proposition}
\newtheorem{corollary}[theorem]{Corollary}
\theoremstyle{definition}

\newtheorem{example}[theorem]{Example}
\newtheorem{question}[theorem]{Question}

\newtheorem{remark}[theorem]{Remark}

\newcommand{\id}{\text{id}}

\newcommand{\Hom}{\text{Hom}}

\newcommand{\Rep}{\text{Rep}}

\newcommand{\Corep}{\text{Corep}}
\newcommand{\Vect}{\text{Vec}}

\newcommand{\g}{\mathfrak{g}}
\newcommand{\h}{\mathfrak{h}}

\newcommand{\ot}{\otimes}

\newcommand{\ben}{\begin{enumerate}}
\newcommand{\een}{\end{enumerate}}

\newcommand{\Lie}{{\text{Lie}}}

\numberwithin{equation}{section}

\begin{document}

\title[Twisting of affine algebraic groups, II] {Twisting of affine algebraic groups, II}

\author{Shlomo Gelaki}
\address{Department of Mathematics, Iowa State University, Ames, IA 50011, USA} \email{gelaki@iastate.edu}

\date{\today}

\keywords{connected nilpotent and unipotent algebraic groups; Hopf $2$-cocycles; cotriangular Hopf algebras; Ore extensions; Weyl algebras}

\begin{abstract}
We use \cite{G} to study the algebra structure of twisted cotriangular Hopf algebras ${}_J\mathcal{O}(G)_{J}$, where $J$ is a Hopf $2$-cocycle for a connected nilpotent algebraic group $G$ over $\mathbb{C}$. In particular, we show that ${}_J\mathcal{O}(G)_{J}$ is an affine Noetherian domain with Gelfand-Kirillov dimension $\dim(G)$, and that if $G$ is unipotent and $J$ is supported on $G$, then ${}_J\mathcal{O}(G)_{J}\cong U(\g)$ as algebras, where $\g={\rm Lie}(G)$. We  also determine the finite dimensional irreducible representations of ${}_J\mathcal{O}(G)_{J}$, by analyzing twisted function algebras on $(H,H)$-double cosets of the support $H\subset G$ of $J$. Finally, we work out several examples to illustrate our results.
\end{abstract}

\maketitle

\tableofcontents

\section{Introduction}

Let $G$ be an affine algebraic group over $\mathbb{C}$, and let $\mathcal{O}(G)$ be the coordinate algebra of $G$. Then $\mathcal{O}(G)$ is a finitely generated commutative Hopf algebra over $\mathbb{C}$. Recall that Drinfeld's twisting procedure produces (new) cotriangular Hopf algebra structures on the underlying coalgebra of $\mathcal{O}(G)$. Namely, if $J\in (\mathcal{O}(G)^{\ot 2})^*$ is a Hopf $2$-cocycle for $G$, then there is a cotriangular Hopf algebra ${}_J\mathcal{O}(G)_{J}$ which is obtained from $\mathcal{O}(G)$ after twisting its ordinary multiplication by means of $J$ and replacing its $R$-form $1\ot 1$ with $J_{21}^{-1}J$.

In categorical terms, Hopf $2$-cocycles for $G$ correspond to tensor structures on the forgetful functor ${\rm Rep}(G)\to \Vect$ of the tensor category $\Rep(G)$ of finite dimensional rational representations of $G$ (see, e.g, \cite{EGNO}).

If ${}_J\mathcal{O}(G)_{J}=\mathcal{O}(G)$ as Hopf algebras then $J$ is called invariant. Invariant Hopf $2$-cocycles for $G$ form a group, which was described completely for connected $G$ \cite[Theorem 7.8]{EG4}. However, if $J$ is not invariant then the situation becomes much more interesting, since the cotriangular Hopf algebra ${}_J\mathcal{O}(G)_{J}$ will be noncommutative. It is thus natural to study the algebra structure and representation theory of ${}_J\mathcal{O}(G)_{J}$ in cases where the classification of Hopf $2$-cocycles for $G$ is known. For example, for finite groups $G$, this was done in \cite[Theorem 3.2]{EG3} and \cite[Theorem 3.18]{AEGN}.
 
The classification of Hopf $2$-cocycles for connected nilpotent algebraic groups $G$ over $\mathbb{C}$ is also known \cite{G}. For example, Hopf $2$-cocycles in the unipotent case are classified by pairs $(H,\omega)$, where $H$ is a closed subgroup of $G$, called the {\em support} of $J$, and $\omega\in \wedge^2{\rm Lie}(H)^*$ is a non-degenerate $2$-cocycle (equivalently, pairs $(\h,r)$, where $\h$ is a quasi-Frobenius Lie subalgebra of ${\rm Lie}(G)$ and $r\in \wedge^2\h$ is a non-degenerate solution to the CYBE (see \ref{qfla})). This was done in \cite[Theorem 3.2]{EG2}, using Etingof--Kazhdan quantization theory \cite{EK1,EK2,EK3}. Later, we extended Movshev's theory on twisting of finite groups \cite{MOV} to the algebraic group case \cite[Section 3]{G}, and generalized the aforementioned classification to connected nilpotent algebraic groups, without using Etingof--Kazhdan quantization theory \cite[Corollary 5.2, Theorem 5.3]{G}.\footnote{We stress however, 
that both the classification of Hopf $2$-cocycles and Movshev's theory for arbitrary affine algebraic groups over $\mathbb{C}$ are still missing (see, e.g., \cite{G} and references therein).}

Thus our goal in this paper is to study the algebra structure and representation theory of the cotriangular Hopf algebras ${}_J\mathcal{O}(G)_{J}$  for connected nilpotent $G$.

The organization of the paper is as follows. In Section \ref{prelim} we recall some basic notions and results used in the sequel. 

In Section \ref{algstrhopf} we consider the cotriangular Hopf algebras ${}_J\mathcal{O}(G)_{J}$ for {\em unipotent} $G$. We first show that ${}_J\mathcal{O}(G)_{J}$ is an iterated Ore extension of $\mathbb{C}$, thus is an affine Noetherian domain with Gelfand-Kirillov dimension $\dim(G)$ (see Corollary \ref{noethdomunipgr}). Secondly, in Theorem \ref{noethdomminunip} we prove that if $J$ is minimal (i.e., $J$ is supported on $G$) then ${}_J\mathcal{O}(G)_{J}\cong U(\g)$ as algebras, where $\g:={\rm Lie}(G)$, while in general, ${}_J\mathcal{O}(G)_{J}$ is a crossed product algebra of ${}_J\mathcal{O}(H)_{J}\cong U(\h)$ and the algebra $_J{}\mathcal{O}(G/H)$, where $H\subset G$ is the support of $J$ and $\h:={\rm Lie}(H)$ (see Theorem \ref{noethdomunipgen}). 

In Section \ref{algstrquot} we analyze twisted function algebras on $(H,H)$-double cosets in {\em unipotent} $G$, and use \cite{G}, to study the quotient algebra ${}_J\mathcal{O}(Z)_{J}$ of ${}_J\mathcal{O}(G)_{J}$ for a double coset $Z$ in $H\backslash G\slash H$. In Theorems \ref{algisomlemmanormal} and \ref{noethdomunip} we show that ${}_J\mathcal{O}(Z)_{J}$ does not contain a Weyl subalgebra if and only if ${}_J\mathcal{O}(Z)_{J}\cong U(\h)$ as algebras, if and only if, $H$ and $J$ are $g$-invariants for $g\in Z$. 

In Section \ref{reprs} we determine the finite dimensional irreducible representations of ${}_J\mathcal{O}(G)_{J}$ (see Theorem \ref{repthm}). Namely, in Theorem \ref{thm1} we show that every irreducible representation of ${}_J\mathcal{O}(G)_{J}$ factors through a unique quotient algebra ${}_J\mathcal{O}(Z)_{J}$, and then deduce from Theorems \ref{algisomlemmanormal} and \ref{noethdomunip} that ${}_J\mathcal{O}(Z)_{J}$ has a finite dimensional irreducible representation if and only if ${}_J\mathcal{O}(Z)_{J}\cong U(\h)$ as algebras, if and only if, $H$ and $J$ are $g$-invariants for $g\in Z$. 

In Section \ref{examples} we give several examples that illustrate the results from Sections \ref{algstrhopf}--\ref{reprs} (see Examples \ref{heisen}--\ref{heisen5}). 

Finally, in Section \ref{nilphopf} we use the results from Sections \ref{algstrhopf}--\ref{reprs} to describe the algebra structure and representations of the cotriangular Hopf algebras ${}_J\mathcal{O}(G)_{J}$ for {\em connected nilpotent} $G$ (see Theorem \ref{noethdomnilp}).  
\medskip

{\bf Acknowledgments.} 
I am grateful to Pavel Etingof for stimulating and helpful discussions. This material is based upon work supported by the National Science Foundation under Grant No. DMS-1440140, while the author was in residence at the Mathematical Sciences Research Institute in Berkeley, California, during the Spring 2020 semester.

\section{Preliminaries}\label{prelim}

\subsection{Hopf $2$-cocycles}\label{hopf2coc} Let $A$ be a Hopf algebra over $\mathbb{C}$. A linear form 
$J:A\ot A\to \mathbb{C}$ is called a Hopf $2$-cocycle for $A$ if it has an inverse $J^{-1}$ under the convolution product $*$ in $\Hom _{\mathbb{C}}(A\ot A,\mathbb{C})$, and satisfies
\begin{align*}\label{2coc}
\sum J (a_1b_1,c)J (a_2,b_2)&=\sum J
(a,b_1c_1)J (b_2,c_2),\\ J (a,1) =\varepsilon(a)&=J (1,a)\nonumber
\end{align*}
for all $a,b,c\in A$.

Given two Hopf $2$-cocycles $K,J$ for $A$, one constructs a new algebra ${}_{K} A_{J}$ as follows. As vector spaces, ${}_{K} A_{J}=A$, and the new multiplication ${}_{K} m_{J}$ is given by
\begin{equation*}\label{nm2}
{}_{K} m_{J}(a\ot b)=\sum K^{-1} (a_1,b_1)a_2b_2J
(a_3,b_3),\,\,\,a,b\in A.
\end{equation*}
 
In particular, ${}_J A_{J}$ is a Hopf algebra,\footnote{${}_J A_{J}$ is denoted also by $A^{J}$, e.g., in \cite{G}.} where ${}_J A_{J}=A$ as coalgebras and the new multiplication ${}_J m_{J}$ is given by
\begin{equation}\label{nm}
{}_J m_{J}(a\ot b)=\sum J^{-1} (a_1,b_1)a_2b_2J
(a_3,b_3),\,\,\,a,b\in A.
\end{equation} 
Equivalently, $J$ defines a tensor structure on the forgetful functor \linebreak $\Corep(A)\to \Vect$.

We also have two new unital associative algebras $A_J:={}_1 A_{J}$ and $_{K} A:={}_{K} A_{1}$, with multiplication rules given respectively by 
\begin{equation}\label{multj}
m_J(a\ot b)=\sum a_1b_1J(a_2,b_2),\,\,\,a,b\in A,
\end{equation}
and
\begin{equation}\label{multjinv}
{}_{K} m(a\ot b)=\sum {K}^{-1}(a_1,b_1)a_2b_2,\,\,\,a,b\in A.
\end{equation}
(For more details, see, e.g, \cite{EGNO}.)

\begin{remark}
The algebras $A_{J}$, ${}_{K} A$ and ${}_{K} A_{J}$ are called $(A,{}_JA_{J})$-biGalois, $({}_KA_{K},A)$-biGalois and $({}_KA_{K},{}_JA_{J})$-biGalois algebras, respectively.
\end{remark} 

\begin{lemma}\label{deltaalgmap}
The comultiplication map $\Delta$ of $A$ determines an injective algebra homomorphism $\Delta:{}_{K} A_{J}\xrightarrow{{\rm 1:1}}
{}_{K} A\ot A_J$. 
\end{lemma}

\begin{proof}
For every $a,b\in A$, we have 
\begin{eqnarray*}
\lefteqn{\Delta(a)\Delta(b)
= \sum {}_{K} m(a_1\ot b_1)\ot m_J(a_2\ot b_2)}\\
& = &\sum {K}^{-1}(a_1,b_1)a_2b_2\ot  a_3b_3J(a_4,b_4)\\
& = & \Delta\left(\sum {K}^{-1} (a_1,b_1)a_2b_2J
(a_3b_3)\right)=\Delta({}_{K} m_{J}(a\ot b)),
\end{eqnarray*}
as claimed.
\end{proof}

\subsection{Cotriangular Hopf algebras}\label{cothopalg} Recall that $(A,R)$ is cotriangular if $R:A\ot A\to \mathbb{C}$ is an invertible linear map under $*$, such that $R^{-1}=R_{21}$, and for every $a,b,c\in A$, we have 
$$R(a,bc)=\sum R(a_1,b)R(a_2,c),\,\,\,R(ab,c)=\sum R(b,c_1)R(a,c_2),$$ 
and
$$\sum R(a_1,b_1)b_2a_2=\sum a_1b_1R(a_2,b_2).$$

Recall that if $R$ is non-degenerate, $(A,R)$ is called minimal, and in this case $R$ defines two injective Hopf algebra maps $A\xrightarrow{1-1} A^*_{\rm fin}$ from $A$ into its finite dual Hopf algebra $A^*_{\rm fin}$. Recall also that any cotriangular Hopf algebra $(A,R)$ has a unique minimal cotriangular Hopf algebra quotient \cite[Proposition 2.1]{G}.

Given a Hopf $2$-cocycle $J$ for $A$, $({}_JA_{J},R^J)$ is also cotriangular, where $R^{J}:=J_{21}^{-1}*R*J$.
(For more details, see, e.g, \cite{EGNO}.)

\begin{lemma}\label{prim1}
Assume $A$ is commutative, and let $J$ be a Hopf $2$-cocycle for $A$. 
If $p\in A$ is primitive (i.e., $\Delta(p)=p\ot 1+1\ot p$) then for every $a\in A$, we have
$$R^J(p,a)=(J-J_{21})(p,a)=(J_{21}^{-1}-J^{-1})(p,a).$$ 
\end{lemma}

\begin{proof}
Since $p(1)=0$, we have
\begin{eqnarray*}
\lefteqn{0=J*J^{-1}(p,a)=\sum J(p_1,a_1)J^{-1}(p_2,a_2)}\\
& = & \sum J(p,a_1)J^{-1}(1,a_2)+\sum J(1,a_1)J^{-1}(p,a_2)\\
& = & (J+J^{-1})(p,a).
\end{eqnarray*}
Thus, we have
\begin{eqnarray*}
\lefteqn{R^J(p,a)=\sum J_{21}^{-1}(p_1,a_1)J(p_2,a_2)}\\
& = & \sum J_{21}^{-1}(p,a_1)J(1,a_2)+\sum J_{21}^{-1}(1,a_1)J(p,a_2)\\
& = & (J+J_{21}^{-1})(p,a)=(J-J_{21})(p,a),
\end{eqnarray*}
as claimed.
\end{proof}

\subsection{Quasi-Frobenius Lie algebras}\label{qfla} Recall that a quasi-Frobenius Lie algebra is a Lie algebra $\h$ equipped with 
a non-degenerate skew-symmetric bilinear form $\omega:\h\times \h\to \mathbb{C}$ satisfying
\begin{equation*}\label{2cocy}
\omega([x,y],z)+\omega([z,x],y)+\omega([y,z],x)=0,\,\,\,x,y,z\in \h
\end{equation*}
(i.e., $\omega$ is a symplectic $2$-cocycle on $\h$). 

Let $\g$ be a Lie algebra. Recall that an element $r\in \wedge^2 \g$ is a solution of the classical Yang-Baxter equation (CYBE) if
$$[r_{12},r_{13}]+[r_{12},r_{23}]+[r_{13},r_{23}]=0.$$
By Drinfeld \cite{D}, solutions $r$ of the CYBE in $\wedge^2\g$
are classified by pairs $(\h,\omega)$, via $r=\omega^{-1}\in \wedge^2\h$, where $\h\subset \g$ is a quasi-Frobenius Lie
subalgebra with symplectic $2$-cocycle $\omega$.

\subsection{Ore extensions}\label{ore} Let $A$ be an algebra, and let  $\delta:A\to A$ be an algebra derivation of $A$. Recall that  the Ore extension $A[y;\delta]$ of $A$ is the algebra generated over $A$ by $y$, subject to the relations $ya-ay=\delta(a)$ for every $a\in A$. (See, e.g., \cite{MR}.)

\subsection{Unipotent algebraic groups}\label{UAG} Let $G$ be a unipotent algebraic group of dimension $m$ over $\mathbb{C}$. Recall that $A:=\mathcal{O}(G)$ is a finitely generated commutative irreducible pointed Hopf algebra, which is isomorphic to a polynomial algebra in $m$ variables as an algebra.
 
Assume we have a central extension
\begin{equation*}\label{exactsequence1}
1\to C\xrightarrow{\iota}G\xrightarrow{\pi}\overline G\to 1,
\end{equation*} 
where $C\cong\mathbb{G}_a$ (= additive group). Then we can view $\mathcal{O}(\overline{G})$ as a Hopf subalgebra of $\mathcal{O}(G)$ via $\pi^*$. Let $\mathcal{O}(C)=\mathbb{C}[z]$, $z$ is primitive. Choose $W$ in $\mathcal{O}(G)$ that maps to $z$ under the surjective Hopf algebra map $\iota^*:\mathcal{O}(G)\twoheadrightarrow \mathcal{O}(C)$, with minimal possible degree with respect to the coradical filtration. Set
$$q(W):=\Delta(W)-W\ot 1 - 1\ot W.$$

\begin{lemma}\label{orehopf1}
$q(W)$ is a coalgebra $2$-cocycle in $\mathcal{O}(\overline G)^+\ot \mathcal{O}(\overline G)^+$. 
\end{lemma}

\begin{proof}
Since the components of $q(W)$ have smaller degree than $W$, and are mapped to elements of degree $\le 1$ in $\mathbb{C}[z]$, it follows that $q(W)$ belongs to $\mathcal{O}(\overline G)^+\ot \mathcal{O}(\overline G)^+$. Finally, $q(W)$ is a coalgebra $2$-cocycle since $(\Delta\ot {\rm id})\Delta(W)=({\rm id}\ot \Delta)\Delta(W)$.
\end{proof}

\begin{lemma}\label{orehopf2}
The polynomial algebra $\mathcal{O}(\overline G)[W]$ has a unique Hopf algebra structure such that $\mathcal{O}(\overline G)$ is a Hopf subalgebra of $\mathcal{O}(\overline G)[W]$, and $\Delta(W)=W\ot 1+1\ot W +q(W)$. Moreover, we have 
$\mathcal{O}(G)\cong \mathcal{O}(\overline G)[W]$ as Hopf algebras. 
\end{lemma}

\begin{proof}
It is clear that $\mathcal{O}(G)\cong \mathcal{O}(\overline G)[W]$ as algebras, and that the Hopf algebra structure is well defined by Lemma \ref{orehopf1}. Finally, it is clear that  
$\mathcal{O}(G)\cong \mathcal{O}(\overline G)[W]$ as Hopf algebras.
\end{proof}

Recall that $G$ is obtained from $m$ successive $1$-dimensional central extensions with kernel $\mathbb{G}_a$. Thus by Lemma \ref{orehopf2}, $A$ admits a filtration 
\begin{equation}\label{filt}
\mathbb{C}=A_0\subset A_1\subset\cdots\subset A_i\subset\cdots\subset A_{m}=A
\end{equation}
by Hopf subalgebras $A_i$, such that for every $1\le i\le m$, $A_i=\mathbb{C}[y_1,\dots,y_i]$ is a polynomial algebra and $q(y_i)$ is a coalgebra $2$-cocycle in $A_{i-1}^+\ot A_{i-1}^+$, with $q(y_1)=q(y_2)=0$. We will sometime write $q(y_i)=\sum Y_i'\ot Y_i''$, and $(\id\ot \Delta)(q(y_i))=\sum Y_i'\ot Y_{i1}''\ot Y_{i2}''$.

Finally, let $H\subset G$ be a closed subgroup of codimension $1$. It is well known that $H$ is normal in $G$, so $G/H\cong \mathbb{G}_a$ as algebraic groups.

\begin{lemma}\label{codim1split}
The exact sequence $1\to H\hookrightarrow G\to G/H\to 1$ splits. 
\end{lemma}

\begin{proof}
It is sufficient to show that the exact sequence of Lie algebras $0\to \mathfrak{h}\hookrightarrow \mathfrak{g}\to \mathfrak{g}/\mathfrak{h}\to 0$ splits. But this is clear since $\mathfrak{h}$ is an ideal of $\mathfrak{g}$ of codimension $1$, so by choosing a splitting of vector spaces $\mathfrak{g}=\mathfrak{h}\oplus \mathbb{C}x$, $x\in \mathfrak{g}/\mathfrak{h}$, we see that $x$ acts on $\mathfrak{h}$ by derivations. This implies the statement.
\end{proof}

\section{The algebra structure of ${}_J\mathcal{O}(G)_{J}$ for unipotent $G$}\label{algstrhopf}

In sections \ref{algstrhopf}--\ref{reprs}, $G$ will denote a {\em unipotent} algebraic group over $\mathbb{C}$ of dimension $m$, and $J$ will be a Hopf $2$-cocycle for $G$ (i.e., for $\mathcal{O}(G)$). 

\subsection{Ring theoretic properties}\label{ringprop} Retain the notation from \ref{UAG}, and let $\cdot$ denote the multiplication in ${}_J A_{J}$. Set $Q:=J-J_{21}$.
 
\begin{lemma}\label{lemmaunip}
The following hold:
\begin{enumerate}
\item
(\ref{filt}) determines a Hopf algebra filtration on ${}_J A_{J}$:
\begin{equation*}\label{filt1}
\mathbb{C}=A_0\subset A_1\subset\cdots \subset {}_J(A_{i})_J\subset\cdots\subset {}_J(A_{m})_J={}_J A_{J}.
\end{equation*}
\item
For every $i$, the Hopf algebra ${}_J(A_i)_J$ is generated by $y_i$ over ${}_J(A_{i-1})_J$.
\item
For every $i,j$, we have
$J(y_i,y_j)+J^{-1}(y_i,y_j)=0$.
\item
For every $j<i$, we have
\begin{eqnarray*}
\lefteqn{y_i\cdot y_j=y_iy_j}\\ 
& + & \sum Y_i'J(Y_i'',y_j) + \sum Y_j'J(y_i,Y_j'')+\sum Y_i'Y_j'J(Y_i'',Y_j'')\\ 
& + & \sum J^{-1}(Y_i',Y_j')J(Y_{i2}'',Y_{j2}'')Y_{i1}''Y_{j1}''.
\end{eqnarray*}
\item
For every $j<i$, we have
\begin{eqnarray*}
\lefteqn{[y_i,y_j]:=y_i\cdot y_j - y_j\cdot y_i}\\ 
& = & \sum Y_i'Q(Y_i'',y_j) + \sum Y_j'Q(y_i,Y_j'')+\sum Y_i'Y_j'Q(Y_i'',Y_j'')\\ 
& + & \sum \left(J^{-1}(Y_i',Y_j')J(Y_{i2}'',Y_{j2}'')-J_{21}^{-1}(Y_i',Y_j')J_{21}(Y_{i2}'',Y_{j2}'')\right)Y_{i1}''Y_{j1}''.
\end{eqnarray*}
Hence, $[y_i,y_j]$ belongs to $A_{i-1}^+$. 
\item
$y_1,y_2$ are central primitives in ${}_J A_{J}$.
\item
The linear map 
$\delta_i:{}_J(A_{i-1})_J \to {}_J(A_{i-1})_J,\,\,s\mapsto [y_i,s]$, is an algebra derivation of ${}_J(A_{i-1})_J$ for every $i$.
\item
For every $i$, ${}_J(A_{i})_J\cong{}_J(A_{i-1})_J[y_i;\delta_i]$ as Hopf algebras.
\end{enumerate}
\end{lemma}

\begin{proof}
(1)--(2) follow from (\ref{nm}) and Lemma \ref{orehopf2} since each $A_{i-1}$ is a Hopf subalgebra of $A_i$.   
(3)--(4) follow from (\ref{nm}) and $\varepsilon(y_i\cdot y_j)=0$. (5)--(6) follow from (4),  
(7) from (5), and (8) from (2) and Lemma \ref{orehopf2}.
\end{proof}

\begin{corollary}\label{noethdomunipgr}
The Hopf algebra ${}_J\mathcal{O}(G)_{J}$ is an affine Noetherian domain with Gelfand-Kirrilov dimension $\dim(G)$.
\end{corollary}

\begin{proof}
It follows from Lemma \ref{lemmaunip}(8), by a simple induction, that ${}_J\mathcal{O}(G)_{J}$ is a finitely generated (i.e., affine) Noetherian domain. The claim about Gelfand-Kirillov dimension follows from \cite[Proposition 8.2.11]{MR} and Lemma \ref{lemmaunip}(8) by simple induction. 
\end{proof}

\begin{remark}\label{2cocsim}
One shows similarly that for every Hopf $2$-cocycle $K$ for $G$,  
the algebra $_{K}{}\mathcal{O}(G)_J$ is an affine Noetherian domain with Gelfand-Kirillov dimension $\dim(G)$.
\end{remark}

\subsection{The minimal case}\label{mincase} Let $H\subset G$ be the support of $J$ (see \cite[Section 3.1]{G}). Then $J$ is a minimal Hopf $2$-cocycle for $H$. Let $\h$ be the Lie algebra of $H$.

\begin{theorem}\label{noethdomminunip}
The $R$-form $R^J$ induces algebra isomorphism 
$$R_+:{}_J\mathcal{O}(H)_{J}\xrightarrow{\cong}U(\h).$$
\end{theorem}

\begin{proof}
Since $({}_J\mathcal{O}(H)_{J},R^J)$ is a minimal cotriangular Hopf algebra, we have an injective homomorphism of Hopf algebras $$R_+: {}_J\mathcal{O}(H)_{J}\to ({}_J\mathcal{O}(H)_{J})^*_{\rm fin},\,\,\,R_+(\alpha)(\beta)=R^J(\beta,\alpha)$$ (see \ref{cothopalg}).
Since $({}_J\mathcal{O}(H)_{J})^*_{\rm fin}={}_J(\mathcal{O}(H)^*_{\rm fin})_J$ (where the right hand side is s twisted coalgebra), we have $({}_J\mathcal{O}(H)_{J})^*_{\rm fin}=\mathcal{O}(H)^*_{\rm fin}=U(\h)\rtimes \mathbb{C}[H]$ as algebras. 

Let $\mathfrak{m}=\mathcal{O}(H)^+$ be the maximal ideal of $1$. We first show
that the image of $R_+$ is contained in the algebra of distributions $\mathcal{O}(H)^*_1=U(\h)$ supported at $1$. Namely, that $R_+(\alpha)$ vanishes on some power of $\mathfrak{m}$ in the algebra $\mathcal{O}(H)$ for every $\alpha\in {}_J\mathcal{O}(H)_{J}$. Indeed, if $\alpha\in {}_J\mathcal{O}(H)_{J}$ has degree $n$ with respect to the coradical filtration of $\mathcal{O}(H)$, then any summand in (a shortest expression of) $\Delta^{n+1}(\alpha)$ has at least one $\varepsilon$ as a tensorand. Since $R^J(\varepsilon,\beta)=R^J(\beta,\varepsilon)=\beta(1)=0$ for every $\beta\in \mathfrak{m}$, it follows that $R_+(\alpha)$ vanishes on some power of $\mathfrak{m}$ in the algebra ${}_J\mathcal{O}(H)_{J}$. But it is clear that every such power contains some power of $\mathfrak{m}$ in the algebra $\mathcal{O}(H)$, as desired. Thus, we have an injective algebra homomorphism $R_+: {}_J\mathcal{O}(H)_{J}\xrightarrow{1:1} U(\h)$.

To show that $R_+$ is surjective, it suffices to show that $\h$ belongs to the image of $R_+$. Indeed, let $V\subset \mathfrak{m}$ be the orthogonal complement of $\mathfrak{m}^2$ (with respect to $R^J$). Then $R_+$ maps $V$ injectively into $\h$ (as $\h=(\mathfrak{m}/\mathfrak{m}^2)^*$), and  since $R^J$ is non-degenerate, $\dim(V)=\dim(\mathfrak{m}/\mathfrak{m}^2)$. Thus $R_+(V)=\h$, as required.
\end{proof}

\begin{corollary}\label{repthmmin}
We have an equivalence ${\rm Rep}({}_J\mathcal{O}(H)_{J})\cong {\rm Rep}(U(\h))$ of abelian categories. 
In particular, ${}_J\mathcal{O}(H)_{J}$ has a unique finite dimensional irreducible representation (as $\h$ is nilpotent). \qed
\end{corollary} 

\begin{remark}\label{twisteduoa}
By \cite[Theorems 4.7,\,5.1]{G}, $\mathcal{O}(H)_J$ and ${}_J \mathcal{O}(H)$ are 
Weyl algebras with left and right action of $H$ by algebra automorphisms, respectively.  
The induced action of $\h$ on $\mathcal{O}(H)_J$ by derivations determines a symplectic $2$-cocycle $\omega\in \wedge^2\h^*$. We have $U^{\omega}(\h)\cong \mathcal{O}(H)_J$ as $H$-algebras, where $H$ acts on $U^{\omega}(\h)$ and $\mathcal{O}(H)_J$ by conjugation and left translations, respectively. Similarly, $U^{-\omega}(\h^{{\rm op}})\cong {}_J \mathcal{O}(H)$ as $H$-algebras. Thus, ${}_J \mathcal{O}(H)\ot \mathcal{O}(H)_J\cong U^{(-\omega,\omega)}(\h^{{\rm op}}\oplus \h)$ as $H$-algebras.

Now since by Lemma \ref{deltaalgmap}, we have an algebra isomorphism
$$\Delta:{}_J\mathcal{O}(H)_{J}\xrightarrow{\cong} ({}_J \mathcal{O}(H)\ot \mathcal{O}(H)_J)^{H},$$
where $h\in H$ acts on ${}_J \mathcal{O}(H)\ot \mathcal{O}(H)_J$ via $\rho_{h}\ot \lambda_h$, it follows that 
$${}_J\mathcal{O}(H)_{J}\xrightarrow{\cong}U^{(-\omega,\omega)}(\h^{{\rm op}}\oplus \h)^{H},$$
as algebras. Thus, by Theorem \ref{noethdomminunip}, we have an algebra isomorphism
$$U(\h)\xrightarrow{\cong}U^{(-\omega,\omega)}(\h^{{\rm op}}\oplus \h)^{H}.$$
\end{remark}

\subsection{The general case}\label{gencase} Recall that $\mathcal{O}(G/H)$ and $\mathcal{O}(H\backslash G)$ are left and right coideal subalgebras of $\mathcal{O}(G)$, respectively. It follows that ${}_J \mathcal{O}(G/H)_J={}_J \mathcal{O}(G/H)$ is a subalgebra of both ${}_J\mathcal{O}(G)_{J}$ and ${}_J \mathcal{O}(G)$.

\begin{lemma}\label{centralsubalg}
The subalgebra $\mathcal{O}(H\backslash G\slash H)\subset {}_J\mathcal{O}(G)_{J}$ is central.
\end{lemma}

\begin{proof}
Follows from $\Delta(\mathcal{O}(H\backslash G\slash H))\subset \mathcal{O}(H\backslash G)\ot \mathcal{O}(G\slash H)$.
\end{proof}

\begin{theorem}\label{noethdomunipgen}
The algebra ${}_J\mathcal{O}(G)_{J}$ is isomorphic to a crossed product algebra \footnote{See, e.g., \cite[Section 7.1]{M}.} $${}_J\mathcal{O}(G)_{J}\cong {}_J \mathcal{O}(G/H)\#_{\sigma}{}_J\mathcal{O}(H)_{J}.$$
That is, ${}_J\mathcal{O}(G)_{J}={}_J \mathcal{O}(G/H)\ot {}_J\mathcal{O}(H)_{J}$ as vector spaces, and for every $\alpha, \tilde \alpha \in \mathcal{O}(G/H)$, $\beta,\tilde \beta\in {}_J\mathcal{O}(H)_{J}$, we have
$$
(\alpha\ot \beta)(\tilde \alpha\ot \tilde \beta)=\alpha (\beta_1\cdot \tilde \alpha)\sigma(\beta_2,\tilde \beta_1)\ot \beta_3\tilde \beta_2.$$
\end{theorem}

\begin{proof}
We have a Hopf quotient $\iota^*:{}_J\mathcal{O}(G)_{J}\twoheadrightarrow {}_J\mathcal{O}(H)_{J}$, with Hopf kernel ${}_J \mathcal{O}(G/H)_J={}_J \mathcal{O}(G/H)$. Thus we have an ${}_J\mathcal{O}(H)_{J}$-extension ${}_J \mathcal{O}(G/H)\subset {}_J\mathcal{O}(G)_{J}$ of algebras. We claim it is cleft. Indeed, choose a regular section $j$ to the inclusion morphism $\iota:H\hookrightarrow G$ (this is possible since $G$ is unipotent). Then $\gamma:=\varphi:{}_J\mathcal{O}(H)_{J}\to {}_J\mathcal{O}(G)_{J}$ is an invertible \footnote{In the convolution algebra $\Hom({}_J\mathcal{O}(H)_{J},{}_J\mathcal{O}(G)_{J})$.} ${}_J\mathcal{O}(H)_{J}$-comodule map, as required. Hence, the statement follows from \cite[Theorem 7.2.2]{M}, where the invertible \footnote{In the convolution algebra $\Hom(({}_J\mathcal{O}(H)_{J})^{\ot 2},{}_J \mathcal{O}(G/H))$.} $2$-cocycle $$\sigma:({}_J\mathcal{O}(H)_{J})^{\ot 2}\to {}_J \mathcal{O}(G/H)$$ 
and twisted action \footnote{I.e., $\beta\cdot(\alpha\tilde \alpha)=(\beta_1\cdot\alpha)(\beta_2\cdot\tilde\alpha)$ and $\beta\cdot(\tilde\beta\cdot \alpha)=\sigma(\beta_1,\tilde\beta_1)(\beta_2\tilde\beta_2\cdot \alpha)\sigma^{-1}(\beta_3,\tilde\beta_3)$.} 
$${}_J\mathcal{O}(H)_{J}\ot {}_J \mathcal{O}(G/H)\to {}_J \mathcal{O}(G/H)$$
are given by 
$$
\sigma(\beta,\tilde \beta)=\gamma(\beta_1)\gamma(\tilde \beta_1)\gamma^{-1}(\beta_2\tilde \beta_2)
$$
and 
$$\beta\cdot\alpha=\gamma(\beta_1)\alpha\gamma^{-1}(\beta_2),$$
for every $\beta,\tilde \beta\in {}_J\mathcal{O}(H)_{J}$ and $\alpha\in \mathcal{O}(G/H)$.
\end{proof}

\begin{remark}\label{typically}
If $H$ is normal in $G$ then ${}_J \mathcal{O}(G/H)=\mathcal{O}(G/H)$ is commutative. However, if $H$ is not normal in $G$ then the algebra ${}_J \mathcal{O}(G/H)$ is typically not commutative (see Example \ref{heisen4}).
\end{remark}

\subsection{One sided twisted algebras}\label{osta} Let $L\subset H$ be a closed subgroup. Since $\mathcal{O}(H/L)$ is a left coideal subalgebra of $\mathcal{O}(H)$, it follows that ${}_J \mathcal{O}(H/L)$ is a subalgebra of ${}_J \mathcal{O}(H)$. Moreover, ${}_J \mathcal{O}(H/L)=({}_J \mathcal{O}(H))^L$ is a subalgebra of the Weyl algebra ${}_J \mathcal{O}(H)\cong U^{\omega}(\h)$ \cite[Theorem 4.7]{G}.

\begin{question}\label{problem1}
What is the algebra structure of ${}_J \mathcal{O}(H/L)$?
\end{question}

We have the following partial answers to Question \ref{problem1}.

Let $\mathscr{W}_n$ denote the Weyl algebra of Gelfand-Kirillov dimension $2n$ ($n\ge 0$). Namely, $\mathscr{W}_n$ is the algebra generated by $2n$ elements $p_1,\dots,p_n$ and $q_1,\dots,q_n$ subject to the relations $[p_i,q_j]=\delta_{i,j}$, $[p_i,p_j]=0$, and $[q_i,q_j]=0$ ($1\le i,j\le n$).

\begin{theorem}\label{paresl}
Let $N\subset H$ be a closed normal subgroup. Then the following hold:
\begin{enumerate}
\item
There exists a closed subgroup $L\subset H$ containing $N$ such that  
$${}_J \mathcal{O}(H/N)\cong \mathcal{O}(L\backslash H)\ot \mathscr{W}_n$$ 
as algebras, where $2n=\dim(L)-\dim(N)$.
\item
If $2\dim(N)<\dim(H)$, ${}_J \mathcal{O}(H/N)$ contains a Weyl subalgebra.
\end{enumerate}
\end{theorem}

\begin{proof}
(1) Since $N$ is normal in $H$, $\mathcal{O}(H/N)$ is a Hopf subalgebra of $\mathcal{O}(H)$. Thus $J$ restricts to a Hopf $2$-cocycle of $H/N$. By \cite[Theorem 3.1]{G}, there exists a closed subgroup $L$ of $H$ containing $N$ such that $L/N\subset H/N$ is the support of $J$. Then by \cite[Theorem 4.7]{G}, we have an algebra isomorphism
$${}_J \mathcal{O}(H/N)\cong \mathcal{O}((L/N)\backslash (H/N))\ot \mathscr{W}_n\cong \mathcal{O}(L\backslash H)\ot \mathscr{W}_n,$$
as claimed.

(2) By (1), it suffices to show that the restriction of $J$ to $\mathcal{O}(H/N)$ is not trivial (since then $n\ge 1$). Suppose otherwise. Then ${}_J\mathcal{O}(H/N)_J$ is isotropic with respect to $R^J$. Thus, $\dim(H/N)\le \dim(H)/2$, which is not the case.
\end{proof}

\begin{corollary}
Let $L\subset H$ be a closed subgroup and let $N$ be its normal closure. If $2\dim(N)<\dim(H)$ then ${}_J \mathcal{O}(H/L)$ contains a Weyl subalgebra, or equivalently, ${}_J \mathcal{O}(H/L)$ is noncommutative. \qed
\end{corollary}

\subsection{$1$-dimensional central extensions}\label{1dce} Suppose we have a central extension
\begin{equation*}\label{exactsequence2}
1\to C\xrightarrow{}G\xrightarrow{\pi}\overline G\to 1,
\end{equation*} 
such that $\mathcal{O}(C)=\mathbb{C}[z]$ ($z$ is primitive), and let $W$ in $\mathcal{O}(G)$ be as in \ref{UAG}. Then by Lemma \ref{lemmaunip}, we have an isomorphism of Hopf algebras  
\begin{equation*}\label{know}
{}_J\mathcal{O}(G)_{J}\cong{}_J\mathcal{O}(\overline G)_{J}[W;\delta],
\end{equation*}
where the derivation $\delta$ is given by $\delta(V)=[W,V]$ for every $V\in {}_J\mathcal{O}(\overline G)_{J}$, and $q(W)$ is in $\mathcal{O}(\overline G)^+\ot \mathcal{O}(\overline G)^+$.

Set $\overline{H}:=\pi(H)$, and let $L\subset \overline H$ be the support of the restriction of $J$ to $\mathcal{O}(\overline{G})$. By \cite[Proposition 4.6]{G}, $L$ has codimension $\le 1$ in $\overline{H}$.

\subsubsection{$C\cap H=\{1\}$}\label{i} In this case,  $L=\overline{H}$. Write
\begin{equation*}\label{comhaunip}
q(W)=\sum W'\ot W''\in \mathcal{O}(\overline G)^+\ot \mathcal{O}(\overline G)^+
\end{equation*}
in the shortest possible way. 
Set $S:=J^{-1}-J_{21}^{-1}$. 

\begin{lemma}\label{choosey}
We can assume that $W\in \mathcal{O}(G\slash H)^+$, and then  
have
$[W,V]=\sum S(W',V_1)W''V_2$ for every $V\in {}_J\mathcal{O}(\overline{G})_{J}$.
\end{lemma}

\begin{proof}
The first claim follows from $C\cap H=\{1\}$. Since $\mathcal{I}(H)$ is a Hopf ideal and $\mathcal{O}(G\slash H)$ is a left coideal in $\mathcal{O}(G)$, $({\rm id}\ot \Delta)\Delta(W)$ lies in $$\mathcal{I}(H)\ot \mathcal{O}(G)\ot \mathcal{O}(G\slash H) + \mathcal{O}(G)^{\ot 2}\ot \mathcal{I}(H) + \mathcal{O}(G)\ot \mathcal{I}(H)\ot \mathcal{O}(G\slash H),$$ which implies the second claim.
\end{proof}

\subsubsection{$C\subset H$}\label{ii} In this case, $L$ has codimension $1$. Let  $A:=\overline{H}/L$. Then $\mathcal{O}(A)=\mathbb{C}[x]$, $x$ is primitive, and   
the quotient map $\sigma:\overline{H}\twoheadrightarrow A$ induces an injective homomorphism of Hopf algebras $\sigma^*:\mathbb{C}[x]\xrightarrow{1:1}{}_J\mathcal{O}(\overline{H})_J$. Thus, we can view $\mathbb{C}[x]$ as a Hopf subalgebra of ${}_J\mathcal{O}(\overline{H})_J$ via $\sigma^*$. Also, by Lemma \ref{codim1split}, we can choose a splitting homomorphism of groups $j:A\xrightarrow{1:1}\overline{H}$ of $\sigma$, and view $A$ as a subgroup of $\overline{H}\subset\overline{G}$ via $j$. Then $j^*:\mathcal{O}(\overline{H})\twoheadrightarrow \mathbb{C}[x]$ is a surjective homomorphism of Hopf algebras.

\begin{lemma}\label{variousgrps}
The Hopf algebra surjective map $\iota_{\overline{H}}^*:\mathcal{O}(\overline{G}) \twoheadrightarrow \mathcal{O}(\overline{H})$ restricts to an algebra surjective map $\iota_{\overline{H}}^*:\mathcal{O}(L\backslash\overline{G}\slash L)\twoheadrightarrow \mathcal{O}(A)$. 
\end{lemma}

\begin{proof}
Clearly, $\iota_{\overline{H}}^*$ maps $\mathcal{O}(L\backslash\overline{G}\slash L)$ onto $\mathcal{O}(L\backslash\overline{H}\slash L)$. Since $L$ is normal in $\overline{H}$, we have $\mathcal{O}(L\backslash\overline{H}\slash L)=\mathcal{O}(A)$, which implies the statement.
\end{proof}

Consider ${}_J\mathcal{O}(\overline{G})_{J}$ as a left comodule algebra over $\mathbb{C}[x]$ via 
$$(j^*\circ\iota_{\overline{H}}^*\ot {\rm id})\circ\Delta:{}_J\mathcal{O}(\overline{G})_{J}\to \mathbb{C}[x]\ot {}_J\mathcal{O}(\overline{G})_{J}.$$ Let $B\subset {}_J\mathcal{O}(\overline{G})_{J}$ be the coinvariant subalgebra. Pick $X\in \mathcal{O}(L\backslash\overline{G}\slash L)$ such that $\iota_{\overline{H}}^*(X)=x$.

\begin{lemma}\label{tendecomp}
The multiplication map $B\ot \mathbb{C}[X]\xrightarrow{}{}_J\mathcal{O}(\overline{G})_{J}$ is an isomorphism of algebras.
\end{lemma}

\begin{proof}
Follows since by Lemma \ref{centralsubalg}, $X$ is central in ${}_J\mathcal{O}(\overline{G})_{J}$.  
\end{proof}

\subsection{The algebra ${}_{J^g}\mathcal{O}(G)_J$}\label{aldjg} Fix $g\in G$. Set ${\rm Ad}g:=\rho_{g}\circ\lambda_{g}$, and $J^g:=J\circ ({\rm Ad}g\ot {\rm Ad}g)$. 

\begin{lemma}\label{adg}
$\lambda_{g^{-1}}:{}_J\mathcal{O}(G)_{J}\to {}_{J^g}\mathcal{O}(G)_J$ is an algebra isomorphism.
\end{lemma} 

\begin{proof}
Straightforward.
\end{proof}

\section{The quotient algebras ${}_J\mathcal{O}(Z_g)_J$}\label{algstrquot} 

Retain the notation of Section \ref{algstrhopf}. 
Every double coset $Z=HgH$ in $H\backslash G\slash H$ is an orbit of the left action of the unipotent algebraic group $H\times H$ on $G$, given by $(a,b)\cdot g:=agb^{-1}$. Hence $Z$ is an irreducible closed subset of $G$ by the theorem of Kostant and Rosenlicht, and it has dimension $2\dim(H)-\dim(H\cap gHg^{-1})$ ($g\in Z$). 

Let $\mathcal{I}(Z)\subset \mathcal{O}(G)$ be the defining ideal of the double coset $Z$. Since $Z$ is irreducible, $\mathcal{I}(Z)$ is a prime ideal. Clearly, $\bigcap_Z \mathcal{I}(Z)=0$.

Now fix a double coset $Z_g=HgH$. Set $H_g:=H\cap gHg^{-1}$, and consider the embedding $$\theta:H_g\to H\times H,\;\;a\mapsto (a,g^{-1}ag),$$ of $H_g$ as a closed subgroup of $H\times H$. The subgroup $\theta(H_g)$ acts on $H\times H$ from the right via
$(x,y)\theta(a)= (xa,g^{-1}a^{-1}gy)$, $x,y\in H$ and $a\in H_g$. Let $\overline{(x,y)}$ denote the orbit of $(x,y)$ under this action. Then we have an isomorphism of affine varieties
$$(H\times H)/\theta(H_g)\xrightarrow{\cong} Z_g,\,\,\,\overline{(x,y)}\mapsto xgy.$$ 

The above right action induces a left action of $\theta(H_g)$ on $\mathcal{O}(H)^{\ot 2}$, given by $(\theta(a)(\alpha\ot \beta)(x,y)=\alpha(xa)\beta(g^{-1}a^{-1}gy)$, where $x,y\in H$ and $a\in H_g$. In other words, the action of $\theta(a)$ is via $\rho_a\ot \lambda_{g^{-1}ag}$, where $\lambda,\rho$ are the left, right regular actions of $G$ on $\mathcal{O}(G)$. Let $(\mathcal{O}(H)\ot \mathcal{O}(H))^{\theta(H_g)}$ denote the subalgebra of invariants under this action. Then we have an algebra isomorphism
$$\mathcal{O}(Z_g)\xrightarrow{\cong} (\mathcal{O}(H)\ot \mathcal{O}(H))^{\theta(H_g)}.$$
Equivalently, the surjective algebra homomorphism 
$$m_g^*:=(\iota^*\ot \iota^*)({\rm id}\ot \lambda_{g^{-1}})\Delta:\mathcal{O}(G)\twoheadrightarrow (\mathcal{O}(H)\ot \mathcal{O}(H))^{\theta(H_g)}$$
has kernel $\mathcal{I}(Z_g)$.

\begin{proposition}\label{lem2}
The map $m_g^*$ determines a surjective algebra homomorphism 
\begin{equation*}\label{twisteddoublecosetalg}
m_g^*:{}_J\mathcal{O}(G)_{J}\twoheadrightarrow (_{J} \mathcal{O}(H)\ot \mathcal{O}(H)_J)^{\theta(H_g)}.
\end{equation*}
In particular, $\mathcal{I}(Z_g)$ is a two sided ideal in ${}_J\mathcal{O}(G)_{J}$.
\end{proposition}

\begin{proof}
Since $\iota^*:{}_J\mathcal{O}(G)_{J}\twoheadrightarrow {}_J\mathcal{O}(H)_{J}$ is an algebra homomorphism, it suffices to show that $({\rm id}\ot \lambda_{g^{-1}})\Delta$ is an algebra homomorphism.
To this end, notice that we have $\Delta\circ\lambda_{g^{-1}}=(\lambda_{g^{-1}}\ot {\rm id})\circ\Delta$. Thus, using (\ref{nm})-(\ref{multjinv}), we get that for every $\alpha,\beta\in \mathcal{O}(G)$,
\begin{eqnarray*}
\lefteqn{({\rm id}\ot \lambda_{g^{-1}})\Delta(\alpha\beta)}\\
& = & ({\rm id}\ot \lambda_{g^{-1}})\Delta\left(\sum J^{-1}(\alpha_1,\beta_1)\alpha_2\beta_2J(\alpha_3,\beta_3)\right)\\
& = & \sum J^{-1}(\alpha_1,\beta_1)\alpha_2\beta_2\ot \lambda_{g^{-1}}(\alpha_3\beta_3)J(\alpha_4,\beta_4)\\
& = & \sum J^{-1}(\alpha_1,\beta_1)\alpha_2\beta_2\ot \lambda_{g^{-1}}(\alpha_3)\lambda_{g^{-1}}(\beta_3)J(\alpha_4,\beta_4)\\
& = & \sum (\alpha_1\ot \lambda_{g^{-1}}(\alpha_2))(\beta_1\ot \lambda_{g^{-1}}(\beta_2))\\
& = & ({\rm id}\ot \lambda_{g^{-1}})\Delta(\alpha)({\rm id}\ot \lambda_{g^{-1}})\Delta(\beta),
\end{eqnarray*}
as required.
\end{proof}

For $g\in G$, set ${}_J\mathcal{O}(Z_g)_{J}:={}_J\mathcal{O}(G)_{J}/\mathcal{I}(Z_g)$.

\begin{corollary}\label{lem3}
For every $g\in G$, the algebra ${}_J\mathcal{O}(Z_g)_{J}$ is an affine Noetherian domain of Gelfand-Kirillov dimension $2\dim(H)-\dim(H_g)$. In particular, $\mathcal{I}(Z_g)$ is a completely prime two sided ideal of ${}_J\mathcal{O}(G)_{J}$.
\end{corollary}

\begin{proof}
Since by \cite[Theorem 4.7]{G}, $_{J} \mathcal{O}(H)\ot \mathcal{O}(H)_J$ is a Weyl algebra, the claim follows from Corollary \ref{noethdomunipgr} and Proposition \ref{lem2}.
\end{proof}

\begin{remark}\label{2cosim2}
Let $K$ be a Hopf $2$-cocycle for $\mathcal{O}(G)$ with support $\widetilde H$. For every $g\in G$, let $Z_g:=\widetilde H gH$ be the $(\widetilde H,H)$-double coset of $g$, let $N_g:=\widetilde H\cap gHg^{-1}$, and let
$$d_g:=\frac{1}{2}\left(\dim(H)+\dim(\widetilde H)\right)-\dim(N_g).$$ (By \cite[Theorem 5.1]{G}, $\dim(H)$ and $\dim(\widetilde H)$ are even, so $d_g$ is an integer.) Then $\mathcal{I}(Z_g)$ is a completely prime two sided ideal of $_{K}{}\mathcal{O}(G)_J$, and $_{K}{}\mathcal{O}(Z_g)_{J}:=(_{K}{}\mathcal{O}(G)_J)/\mathcal{I}(Z_g)$ is an affine Noetherian domain with Gelfand-Kirillov dimension $\dim(Z_g)=\dim(H)+\dim(\widetilde H)-\dim(N_g)$.
\end{remark}

\begin{remark}
By Proposition \ref{lem2} and \cite[Theorem 4.7]{G}, if $H_g$ is trivial then ${}_J\mathcal{O}(Z_g)_{J}\cong \mathscr{W}_{\dim(H)}$ is a Weyl algebra.
\end{remark}

\begin{theorem}\label{algisomlemmanormal}
Assume $H$ is $g$-invariant, and let $\omega_g:=\omega^g - \omega$. 
Then we have an algebra isomorphism
$${}_J\mathcal{O}(Z_g)_{J}\cong {}_{J^g}\mathcal{O}(H)_J\cong U^{\omega_g}(\h).$$
In particular, a maximal Weyl subalgebra of ${}_J\mathcal{O}(Z_g)_{J}$ has dimension $r$, where $r\in 2\mathbb{Z}^{\ge 0}$ is the rank of $\omega_g$ restricted to $\mathfrak{h}$.
\end{theorem}

\begin{proof}
Since $H$ is $g$-invariant if and only if ${\rm Ad}g$ defines a Hopf algebra isomorphism $\mathcal{O}(H)\xrightarrow{\cong}\mathcal{O}(H)$, $J^g$ is a minimal Hopf $2$-cocycle for $H$. Clearly, $J^g$ corresponds to the symplectic $2$-cocycle $\omega^g$ of $\h$.

Now since $\lambda_{g^{-1}}$ maps $\mathcal{I}(gH)$ isomorphically onto $\mathcal{I}(H)$, it follows from Lemma \ref{adg} that $\lambda_{g^{-1}}$ induces an algebra isomorphism
$$\lambda_{g^{-1}}:{}_J\mathcal{O}(Z_g)_{J}\cong {}_J\mathcal{O}(G)_{J}\slash \mathcal{I}(gH)\xrightarrow{\cong} {}_{J^g}\mathcal{O}(G)_J\slash \mathcal{I}(H).$$

Finally, by Theorem \ref{noethdomminunip}, we have algebra isomorphisms
$${}_{J^g}\mathcal{O}(G)_J\slash \mathcal{I}(H)\cong {}_{J^g}\left(\mathcal{O}(G)\slash \mathcal{I}(H)\right)_J\cong {}_{J^g}\mathcal{O}(H)_J\cong U^{\omega_g}(\h),$$
as desired.
\end{proof}

Next we consider the case where $H$ is not $g$-invariant, i.e., the case $d:=\dim(H/H_g)>0$.

\begin{theorem}\label{noethdomunip}
If $H$ is not $g$-invariant then the algebra $\mathcal{O}(Z_g)^{J}$ contains a Weyl subalgebra.
\end{theorem}

\begin{proof} 
The proof is by induction on the dimension $m$ of $G$, $m\ge 4$.
 
Assume $m=4$. Since $G$ is not commutative, $\dim(H)=2$. Since $H$ is properly contained in a proper normal subgroup $N$ of $G$, it follows that $N$ is the Heisenberg group of dimension $3$. Thus, the induction base is given in Example \ref{heisen1}.

Assume $m\ge 5$. Since $G$ is unipotent, we have a central extension 
$$1\to C\xrightarrow{\iota}G\xrightarrow{\pi}\overline G\to 1,$$ where $C\cong \mathbb{G}_a$
as in \ref{1dce}. Let $\mathcal{O}(C)=\mathbb{C}[z]$, $W\in \mathcal{O}(G)$ and $L\subset \overline H\subset \overline G$ be as in \ref{1dce}. Set $\bar g:=\pi(g)$ and $\bar d:=d_{\bar g}$. 

There are two possible cases: Either $H\cap C$ is trivial or $C\subset H$.

\subsection{$H\cap C=\{1\}$}\label{I} In this case,  
we are in the situation of \ref{i}. 
Consider the regular surjective map $\pi:HgH\twoheadrightarrow L\bar g L$. We have $\pi^{-1}(\bar g)=gC\cap HgH$, and $\dim(HgH)-\dim(L\bar g L)=\dim(\pi^{-1}(\bar g))$.

\begin{lemma}\label{fiber}
Exactly one of the following holds:
\begin{enumerate}
\item
$\pi^{-1}(\bar g)=\{g\}$. Equivalently, $\dim(HgH)=\dim(L\bar g L)$.
\item
$\pi^{-1}(\bar g)=gC$. Equivalently, $\dim(HgH)=\dim(L\bar g L)+1$.
\end{enumerate} 
\end{lemma}

\begin{proof}
Follows since $gC\cap HgH\subset gC$ and $\dim(gC)=1$.
\end{proof}

If Lemma \ref{fiber}(1) holds, then $\mathcal{O}(Z_g)^{J}\cong \mathcal{O}(Z_{\bar g})^{J}$ and $d=\bar d$, so the claim follows by induction.
 
Otherwise, Lemma \ref{fiber}(2) holds. Then $d=\bar d+1$. If ${\bar d}>0$ then $\mathcal{O}(Z_{\bar g})^{J}$ contains a Weyl subalgebra by the induction assumption, and since $\mathcal{O}(Z_{\bar g})^{J}$ is a subalgebra of $\mathcal{O}(Z_g)^{J}$, so does $\mathcal{O}(Z_g)^{J}$.

Otherwise ${\bar d}=0$. So, $d=1$. Thus $L$ is $\bar g-$invariant, $H_g$ is normal in $H$, $HgH=CgH$, and $\dim(L)=\dim(H_g)+1$. 

\begin{lemma}\label{wonhgh}
The following hold:
\begin{enumerate}
\item
$W$ is not constant on $HgH=CgH$.
\item  
$\rho_g(W)-W-W(g)$ vanishes on $C$ and $H_g$, but not on $H$.
\end{enumerate}
\end{lemma}

\begin{proof}
Since $W=z$ on $C$, each $W'$ vanishes on $C$. Hence by Lemma \ref{choosey},  
$W(cgh)=W(cg)=W(c)+W(g)+\sum W'(c)W''(g)=c+W(g)$ 
for every $c\in C$ and $h\in H$, which implies (1) and the first part of (2). Since $\rho_g(W)(ghg^{-1})=W(gh)=W(g)$ for every $ghg^{-1}$ in $H_g$, and $W$ vanishes on $H$, the second part of (2) follows too.
\end{proof}

For $l\in L$, let $\tilde{l}\in L$ be the unique element such that $l=\bar g \tilde{l} {\bar g}^{-1}$. Let $h,\tilde{h}\in H$ be the unique elements such that $l=\pi(h)$ and $\tilde{l}=\pi(\tilde{h})$. Let $\tau(l):=g\tilde{h}g^{-1}h^{-1}$. Then $\tau(l)\in C$.

\begin{lemma}\label{splexse}
$\tau:L\to C$ is a group homomorphism, and we have a splitting exact sequence of algebraic groups
$1\to H_g\xrightarrow{\pi}L\xrightarrow{\tau} C\to 1$.
\end{lemma}

\begin{proof}
Follows from Lemma \ref{codim1split} since $H_g$ has codimension $1$ in $H$.
\end{proof}
  
By Lemma \ref{splexse}, we have an injective homomorphism of Hopf algebras $\tau^*:\mathcal{O}(C)\xrightarrow{1:1}\mathcal{O}(L)$. Let $p:=\tau^*(z)$. Then $p$ is a nonzero primitive element in $\mathcal{O}(L)$ that generates the defining ideal of $\pi(H_g)$ in $\mathcal{O}(L)$.

\begin{lemma}\label{zphow}
We may assume $\iota_L^*(\rho_g(W)-W-W(g))=p$. 
\end{lemma}

\begin{proof}
Consider the surjective Hopf algebra map $f:\mathcal{O}(G)\twoheadrightarrow \mathcal{O}(CH)$ induced by the inclusion of groups $CH\subset G$. Since $W$ vanishes on $H$, and restricts to $z$ on $C$, it follows that $f(W)=z$, and $f(W')$ is a primitive element in $\mathcal{O}(CH)$ that vanishes on $C$ for every $W'$. Thus, $\iota_L^*(\rho_g(W)-W-W(g))=\sum \iota_L^*(W')W''(g)$ is a nonzero primitive element in $\mathcal{O}(L)$ by Lemma \ref{wonhgh}, and since it vanishes on the defining ideal of $\pi(H_g)$ in $\mathcal{O}(L)$, it must be proportional to $p$.  
\end{proof}

Since $R^J$ is non-degenerate on $\mathcal{O}(L)^J$, there exists $X\in \mathcal{O}(\overline{G})$ such that $R^J(p,X)=1$. Choose such $X$ with minimal possible degree with respect to the coradical filtration, and write 
$q(X)=\sum_i X_i\ot Y_i$
in the shortest possible way. Then $R^J(p,X_i)=0$ for every $i$.

\begin{proposition}\label{weyl1}
We have $[W,X]\equiv 1$ on $HgH$. 
\end{proposition}

\begin{proof}   
Since each $W''$ is in $\mathcal{O}(\overline{G}\slash L)$, and $R^J(p,X_i)=0$ for every $i$, we have by Lemma \ref{choosey}, 
\begin{eqnarray*}
\lefteqn{[W,X](\bar g l)=\sum S(W',X_1)W''(\bar g l)X_2(\bar g l)=\sum S(W',X_1)W''(g)X_2(\bar g l)}\\
& = &\sum S(W',X)W''(g)+\sum_i S(W',X_i)W''(g)Y_i(\bar g l)\\
& = & S(\rho_g(W)-W-W(g),X)+\sum_i S(\rho_g(W)-W-W(g),X_i)Y_i(\bar g l)\\
& = & S(p,X)+\sum_i S(p,X_i)Y_i(\bar g l)=R^J(p,X)+\sum_i R^J(p,X_i)Y_i(\bar g l)\\
& = & R^J(p,X)
\end{eqnarray*}
for every $l\in L$, where the equality before last follows from Lemma \ref{prim1}. Thus $[W,X]\equiv 1$ on $HgH$, as claimed.
\end{proof}
 
\subsection{$C\subset H$}\label{II} In this case, we are in the situation of \ref{ii}, and $W$ does not vanish on $H$. 

\subsubsection{$A$ is $\bar g$-invariant}\label{II.1} In this case, we have $\overline{H}\bar g \overline{H}=AL\bar g L=L\bar g LA$, and $\bar d=d>0$.

\begin{proposition}\label{finally}
The algebra ${}_J\mathcal{O}(Z_g)_{J}$ contains a Weyl subalgebra.
\end{proposition}

\begin{proof}
Since ${}_J\mathcal{O}(\overline{H}\bar g \overline{H})_J$ is a subalgebra in ${}_J\mathcal{O}(Z_g)_{J}$ via $\pi^*$, it suffices to show that ${}_J\mathcal{O}(\overline{H}\bar g \overline{H})_J$ contains a Weyl subalgebra.

Now since $\bar d>0$, the algebra    
${}_J\mathcal{O}(Z_{{\bar g}})_J$ contains a Weyl subalgebra by the induction assumption. Let $\alpha,\beta$ in ${}_J\mathcal{O}(\overline{G})_{J}$ such that $[\alpha,\beta]\equiv 1$ on $L\bar g L$. By Lemma \ref{tendecomp}, we can write $\alpha=\sum_i\alpha_iX^i$ and $\beta=\sum_i\beta_iX^i$, where $\alpha_i,\beta_i$ are in $B$, and $[\alpha,\beta]=\sum_{i,j}[\alpha_i,\beta_j]X^{i+j}$.

Since $X$ is $L$-biinvariant, we have $X\equiv X(\bar g)$ on $L\bar g L$. 
If $X(\bar g)=0$, then $X^i(\bar g)=0$ for all $i\ge 1$, hence $[\alpha_0,\beta_0]=[\alpha,\beta]\equiv 1$ on $L\bar g L$. But $[\alpha_0,\beta_0]$ is in $B$ (as $\alpha_0,\beta_0$ are), so $[\alpha_0,\beta_0]$ is $A$-invariant. Thus, $[\alpha_0,\beta_0]\equiv 1$ on $\overline{H}\bar g \overline{H}=AL\bar g L=L\bar g LA$, and we are done.

Otherwise, $X(\bar g)\ne 0$. We may assume $X(\bar g)=1$. Then we have $\sum_{i,j}[\alpha_i,\beta_j]=[\alpha,\beta]\equiv 1$ on $L\bar g L$. 
Set $\tilde\alpha:=\sum_i\alpha_i$ and $\tilde\beta:=\sum_i\beta_i$. Then $\tilde\alpha,\tilde\beta$ are in $B$, and we have $[\tilde\alpha,\tilde\beta]\equiv 1$ on $L\bar g L$, hence on $\overline{H}\bar g \overline{H}$, as above.
\end{proof}

\subsubsection{$A$ is not $\bar g$-invariant}\label{II.2} In this case, $A\bar g A$ is $2$-dimensional and $d=\bar d+1$. Set $\varphi:=j^*\circ\iota_L^*:\mathcal{O}(\overline{G})\twoheadrightarrow \mathcal{O}(A)$.

\begin{lemma}\label{vexists}
There exists $V\in \mathcal{O}(\overline{G})^+$ such that 
the following hold:
\begin{enumerate}
\item
$\varphi(\rho_{{\bar g}}(V))$ and $\varphi(\lambda_{{\bar g}^{-1}}(V))$ are algebraically independent, and $V(\bar g)=0$. In particular, $V$ is not primitive.
\item
$\varphi(q(V))=\varphi(q(V))_{21}\ne 0$. In particular,    
$\varphi(V)$ is not primitive.
\end{enumerate}
\end{lemma}

\begin{proof}
(1) Since $A$ is not $\bar g$-invariant, the first claim follows, and replacing $V$ by $V-V(\bar g)$ if necessary, we may assume $V(\bar g)=0$. Since either $\varphi(\rho_{{\bar g}}(V))$ or $\varphi(\lambda_{{\bar g}^{-1}}(V))$ is not primitive, $V$ is not primitive.

(2) The first claim follows from $\Delta(\varphi(V))=\Delta^{{\rm op}}(\varphi(V))$ (as $A$ is commutative). If $\varphi(q(V))=\varphi(q(V))_{21}=0$, then $\varphi(\rho_{{\bar g}}(V))=\varphi(V)$ and $\varphi(\lambda_{{\bar g}^{-1}}(V))=\varphi(V)$, which is a contradiction. Thus, $\varphi(V)$ is not primitive, as claimed.
\end{proof}

Pick $V\in \mathcal{O}(\overline{G})^+$ as in Lemma \ref{vexists}, with minimal possible degree $\ell\ge 2$ with respect to the coradical filtration. By Lemma \ref{variousgrps}, we may assume $V\in \mathcal{O}(L\backslash\overline{G}\slash L)^+$.
Write $q(V)=\sum_i X_i\ot Y_i$. Then for every $i$, we have $X_i\in \mathcal{O}(L\backslash\overline{G})^+$ and $Y_i\in \mathcal{O}(\overline{G}\slash L)^+$.

\begin{lemma}\label{import}
We have $[W,V]=\sum_i S(W,X_i)Y_i-\sum_i S(W,Y_i)X_i$.
\end{lemma}

\begin{proof} 
Since $Y_i$ and $V$ lie in $\mathcal{O}(\overline{G}\slash L)$, it follows from (\ref{nm}) that 
$$[W,V]=\sum_i S(W,X_i)Y_i-\sum_i S(W,Y_i)X_i+\sum S(W',X_i)W''Y_i.$$
Moreover, since $X_i\in\mathcal{O}(L\backslash\overline{G})^+$ for every $i$, and $W'\in\mathcal{O}(\overline{G})$, we have $S(W',X_i)=0$ for every $i$ and $W'$, so the claim follows.
\end{proof}

\begin{lemma}\label{propsofv}
$q(V)=X\ot Y$, where $X$ and $Y$ are primitive elements in the defining ideal of $L$, and $\iota_H^*(X)=\iota_H^*(Y)$.
\end{lemma}

\begin{proof}
Suppose $\iota_{\overline{H}}^*(X_i)\ne 0$. Then since $X_i$ vanishes on $L$, it cannot vanish on $A$. So, $\varphi(X_i)\ne 0$. Moreover, since the degree of $X_i$ is $<\ell$, $X_i$ must be primitive by minimality of $\ell$. Similarly, if $\iota_{\overline{H}}^*(Y_i)\ne 0$ then $Y_i$ must be primitive. Thus $\ell=2$, which implies the statement.
\end{proof}

Set $p:=\iota_H^*(X)=\iota_H^*(Y)$. Then $p$ is primitive in $\mathcal{O}(H)$.

\begin{lemma}\label{propsofv2}
We have $[W,V]=R^J(W,p)(Y-X)$.
\end{lemma}

\begin{proof}
By Lemmas \ref{import} and \ref{propsofv}(2), we have
\begin{eqnarray*}
\lefteqn{[W,V]=S(W,\iota_H^*(X))Y-S(W,\iota_H^*(Y))X}\\
& = & S(W,\iota_H^*(X))Y-S(W,\iota_H^*(X))X=S(W,p)(Y-X)\\
& = & R^J(W,p)(Y-X),
\end{eqnarray*}
as claimed, where the last equation holds by Lemma \ref{prim1}.
\end{proof}

Set $c:=R^J(W,p)(Y-X)(g)\in \mathbb{C}$.

\begin{proposition}\label{centralzw}
We have $[W,V]\equiv c\neq 0$ on $HgH$. Thus, we may assume $[W,V]\equiv 1$ on $HgH$.
\end{proposition}

\begin{proof}
We first show that $c\neq 0$. To this end, we have to show that $X(g)\neq Y(g)$ and $R^J(W,p)\neq 0$. Since 
$\varphi(\rho_{{\bar g}}(V))=\varphi(V)+\varphi(X)Y(g)$, $\varphi(\lambda_{{\bar g}^{-1}}(V))=\varphi(V)+\varphi(Y)X(g)$, 
$\varphi(\rho_{{\bar g}}(V))\neq \varphi(\lambda_{{\bar g}^{-1}}(V))$ by Lemma \ref{vexists}, and $\varphi(X)=\varphi(Y)$ by Lemma \ref{propsofv}, we have $X(g)\neq Y(g)$. Furthermore, since $p$ vanishes on $L$ by Lemma \ref{propsofv}, $p$ is orthogonal to ${}_J\mathcal{O}(\overline{H})_J$ inside ${}_J\mathcal{O}(H)_{J}$. Thus, $R^J(W,p)\neq 0$ by the non-degeneracy of $R^J$ on ${}_J\mathcal{O}(H)_{J}$.

Next we show that $[W,V]\equiv c$ on $\overline{H}\bar g\overline{H}$.  
Since by Lemma \ref{propsofv}, $X$ is primitive in $\mathcal{I}(L)$, we have 
$X(a_1l_1\bar g a_2l_2)=X(a_1)+X(\bar g)+X(a_2)$ for every $a_1,a_2\in A$ and $l_1,l_2\in L$, and similarly for $Y$. 
Thus, since by Lemma \ref{propsofv}, $X=Y$ on $A$, we have $(Y-X)(a_1l_1\bar g a_2l_2)=(Y-X)(g)$
for every $a_1,a_2\in A$ and $l_1,l_2\in L$, which implies that $[W,V]\equiv c$ on $\overline{H}\bar g\overline{H}$, as claimed.
\end{proof}
 
The proof of Theorem \ref{noethdomunip} is complete.
\end{proof}

\begin{question}\label{problem2}
Fix $g\in G$, and set $A:={}_J\mathcal{O}(Z_g)_{J}$. 
\begin{enumerate}
\item
Is it true that $A\cong U^{\omega_g}(\mathfrak{h}_g)\ot \mathscr{W}$ as algebras, where $\mathscr{W}$ is a Weyl subalgebra with $\text{GKdim}(\mathscr{W})=2d_g$?
\item
Is it true that $A$ contains a subalgebra $\mathscr{U}\cong U^{\omega_g}(\mathfrak{h}_g)$, and a Weyl subalgebra $\mathscr{W}$ with $\text{GKdim}(\mathscr{W})=2d_g$, such that $\mathscr{U}\cap \mathscr{W}$ is trivial? 
\item
Is it true that a maximal Weyl subalgebra of $A$ has Gelfand-Kirillov dimension $2d_g+r$, where $r\in 2\mathbb{Z}^{\ge 0}$ is the rank of $\omega_g$ restricted to $\mathfrak{h}_g$? 
\end{enumerate}
(See, e.g., the end of Example \ref{heisen5}.)
\end{question}

\begin{remark}
By Proposition \ref{lem2}, Question \ref{problem2} is a special case of Question \ref{problem1}.
\end{remark}

\section{Representations of ${}_J\mathcal{O}(G)_{J}$ for unipotent $G$}\label{reprs} 

Retain the notation of Sections \ref{algstrhopf} and \ref{algstrquot}. 

\begin{theorem}\label{thm1}
Every irreducible representation $V$ of ${}_J\mathcal{O}(G)_{J}$ factors through a unique quotient ${}_J\mathcal{O}(Z)_{J}$.
\end{theorem}

\begin{proof}
By Lemma \ref{centralsubalg}, the central subalgebra $\mathcal{O}(H\backslash G\slash H)\subset {}_J\mathcal{O}(G)_{J}$ acts on $V$ by a certain central character $\chi_0: \mathcal{O}(H\backslash G\slash H)\to \mathbb{C}$. Let $K_0:={\rm Ker}(\chi_0)$, let $I_0\subset \mathcal{O}(G)$ be the ideal generated by $K_0$, and let $Z_0\subset G$ be the closed subscheme defined by $I_0$. Then $Z_0$ is an affine scheme of finite type (i.e., $\mathcal{O}(Z_0)$ can be non-reduced and have nilpotents) with an $H\times H$-action, and all orbits of $H\times H$ on the underlying variety of $Z_0$ are closed by the theorem of Kostant and Rosenlicht since $H$ is unipotent. Pick an orbit $Y$ in $Z_0$. If $Y=Z_0$ then $Z_0=HgH$ is a single $H\times H$-orbit, so $V$ factors through ${}_J\mathcal{O}(Z_g)_{J}$, and we are done. 

Otherwise, let $\tilde{I}_0$ be the ideal of functions on $Z_0$ vanishing on $Y$. Then $\tilde{I}_0$ is invariant under $H\times H$, and is a union of finite dimensional $H\times H$-modules, so it has a fixed vector $f\ne 0$ (as $H$ is unipotent), and this $f$ cannot be constant since it vanishes on $Y$. Thus, $\mathcal{O}(H\backslash Z_0\slash H)$ is nontrivial. 

Now consider the nontrivial central subalgebra $\mathcal{O}(H\backslash Z_0\slash H)\subset \mathcal{O}(Z_0)$. It has a maximal ideal $\mathfrak{n}$, so $\mathcal{O}(H\backslash Z_0\slash H)/\mathfrak{n}$ is a field extension of $\mathbb{C}$. But it is countably dimensional, so has to be $\mathbb{C}$. Thus, we have a central character $\chi_1: \mathcal{O}(H\backslash Z_0\slash H)\to \mathbb{C}$ by which $\mathcal{O}(H\backslash Z_0\slash H)$ acts on $V$, as above. Let $K_1:={\rm Ker}(\chi_1)$, let $I_1\subset \mathcal{O}(G)/I_0$ be the ideal generated by $K_1$, and let $Z_1\subset Z_0$ be the closed subscheme defined by $I_1$. Then $Z_1$ is $H\times H$-stable. Thus we can proceed as above by looking at the orbits of $H\times H$ in $Z_1$. However, by the Hilbert basis theorem, the sequence $Z_0\supset Z_1\supset \cdots$ must stabilize. Hence the process will end, and we will reach a single $H\times H$-orbit $HgH$, as desired.

Finally, since $\mathcal{I}(Z)+\mathcal{I}(Z')={}_J\mathcal{O}(G)_{J}$ for every two distinct double cosets $Z$ and $Z'$, $\mathcal{I}(Z)$ and $\mathcal{I}(Z')$ cannot both annihilate $V$.
\end{proof}
 
Let $N_G(H,J)$ be the subgroup of the normalizer $N_G(H)$ of $H$ in $G$, consisting of all elements $g\in N_G(H)$ such that $J$ is $g$-invariant.

\begin{theorem}\label{repthm}
There is one to one correspondence between 
isomorphism classes of finite dimensional irreducible representations of ${}_J\mathcal{O}(G)_{J}$ and elements of the group $N_G(H,J)$.
\end{theorem}

\begin{proof} 
Follows from Theorems \ref{algisomlemmanormal}, \ref{noethdomunip} and \ref{thm1}, and the facts that Weyl algebras of degree $\ge 1$ have no finite dimensional representations, and nilpotent Lie algebras have only the trivial finite dimensional irreducible representation.
\end{proof}

\begin{remark}\label{2cocsim3}
Retain the notation from Remark \ref{2cosim2}. 
Then every irreducible representation of the algebra $_{K}{}\mathcal{O}(G)_J$ factors through a unique quotient algebra $_{K}{}\mathcal{O}(Z_g)_{J}$. Furthermore, if $d_g=0$ then $K$ and $J$ are gauge equivalent, and $_{K}{}\mathcal{O}(G)_J$ has finite dimensional irreducible representations if and only if $J$ is $g$-invariant.
Indeed, note that  
$d_g=0$ if and only if $N_g=\widetilde H=gHg^{-1}$. But the later implies that $K$, $J$ are gauge equivalent. Thus we are reduced to Theorem \ref{repthm}.
\end{remark}

\section{Examples}\label{examples}

Let $A=\mathbb{G}_a^2$ with $\mathcal{O}(A)=\mathbb{C}[X,V]$, and let $\mathfrak{a}$ be the Lie algebra of $A$, with basis $\{\frac{\partial}{\partial X},\frac{\partial}{\partial V}\}$. Let $r:=\frac{\partial}{\partial X}\wedge \frac{\partial}{\partial V}$. Then the composition
$$J:\mathcal{O}(A)\ot \mathcal{O}(A)\xrightarrow{e^{r/2}}\mathcal{O}(A)\ot \mathcal{O}(A)\xrightarrow{\varepsilon\ot \varepsilon}\mathbb{C}$$  
is a minimal Hopf $2$-cocycle for $A$ \cite[Section 4]{EG1}, and it is straightforward to verify that 
\begin{equation}\label{ersq}
J(X,V)=J^{-1}(V,X)=1/2,\,\,J(V,X)=J^{-1}(X,V)=-1/2.
\end{equation}
Clearly, we have ${}_J\mathcal{O}(A)_{J}=\mathcal{O}(A)\cong U(\mathfrak{a})$ as Hopf algebras.

\begin{example}\label{heisen}
Let 
$$G= \left \{
1 + xE_{12}+ vE_{13} + yE_{23}
\middle | x, v, y\in \mathbb{C}
\right \}\subset U_3
$$
be the Heisenberg group of dimension $3$. Its coordinate Hopf algebra is a polynomial algebra $\mathcal{O}(G)=\mathbb{C}[X,Y,V]$, with $X,Y$ being primitive, and 
$\Delta(V)=V\ot 1 + 1\ot V +X\ot Y$.

Set $$a:=\frac{\partial}{\partial X},\;b:=\frac{\partial}{\partial Y},\;c:=\frac{\partial}{\partial V}.$$ Then $\mathfrak{g}:=\Lie(G)$ has basis $a,b,c$, with bracket $[a,b]=c$. The element $r:=a\wedge c$ is a non-degenerate $\g$-invariant solution to the CYBE in $\wedge^2 \mathfrak{h}$, where $\mathfrak{h}\subset \g$ is the (abelian) Lie subalgebra spanned by $a,c$.  
Thus, $J:=e^{r/2}$ is a minimal Hopf $2$-cocycle for $H$, where 
$$H=\left \{
1 + xE_{12}+ vE_{13}
\middle | x, v\in \mathbb{C}
\right \}\subset G$$ 
is the (normal abelian) subgroup with $\Lie(H)=\mathfrak{h}$, and we have that ${}_J\mathcal{O}(H)_{J}=\mathcal{O}(H)\cong U(\mathfrak{h})$ as Hopf algebras.

We now view $J$ as a (non-minimal) Hopf $2$-cocycle for $G$ by pulling it back along the obvious Hopf algebra surjective map $\mathcal{O}(G)\twoheadrightarrow \mathcal{O}(H)$ determined by $Y\mapsto 0$, $X\mapsto X$, and $V\mapsto V$. 
Since $J$ is an invariant Hopf $2$-cocycle for $G$, 
${}_J\mathcal{O}(G)_{J}=\mathcal{O}(G)\cong \mathcal{O}(G/H)\ot {}_J\mathcal{O}(H)_{J}$ as algebras (so, ${}_J\mathcal{O}(G)_{J}$ is isomorphic to $U(\mathbb{C}^3)$ as an algebra, but not to $U(\g)$), and ${}_J\mathcal{O}(Z_g)_{J}\cong{}_J\mathcal{O}(H)_{J}=\mathcal{O}(H)\cong U(\mathfrak{h})$ for every $g\in G$ since $H$ is normal and $r$ is $\g$-invariant.
\end{example}

\begin{example}\label{heisen1} (The induction base in the proof of Theorem \ref{noethdomunip}.)
Let 
$$G= \left \{
\left( \begin{array}{cccc}
1 & x & v & w \\
0 & 1 & y & \frac{y^2}{2} \\
0 & 0 & 1 & y\\ 
0 & 0 & 0 & 1
\end{array} \right)
\middle | x, v, w, y\in \mathbb{C}
\right \}\subset U_4.
$$
Then $G$ is a $4$-dimensional unipotent algebraic group over $\mathbb{C}$.
Its coordinate Hopf algebra $\mathcal{O}(G)=\mathbb{C}[X,Y,V,W]$ is a polynomial algebra, with $X,Y$ being primitive, and 
$$\Delta(V)=V\ot 1 + 1\ot V +X\ot Y,$$ $$\Delta(W)=W\ot 1 + 1\ot W +V\ot Y + X\ot Y^2/2.$$ 

Let $C:=\{1+wE_{14}\mid w\in \mathbb{C}\}$. Then $C\cong \mathbb{G}_a$ is a closed central subgroup of $G$.

Set $$a:=\frac{\partial}{\partial X},\;b:=\frac{\partial}{\partial Y},\;c:=\frac{\partial}{\partial V},\;d:=\frac{\partial}{\partial W}.$$ Then $\mathfrak{g}:=\Lie(G)$ has basis $a,b,c,d$, with brackets $[a,b]=c$, $[c,b]=d$.

Let 
$$H:=\{1+xE_{12}+vE_{13}\mid x,v\in \mathbb{C}\}\subset G.$$ Then $H\cong\mathbb{G}_a^2$ and $\mathcal{O}(H)=\mathbb{C}[X,V]$ is a polynomial Hopf algebra. Since $C\cap H=\{1\}$, we have $H\cong L$ (see \ref{I}). Let $\mathfrak{h}:={\rm Lie}(H)$ with basis $a,c$, let $r:=a\wedge c$, and let $J:=e^{r/2}$ as above. We have ${}_J\mathcal{O}(H)_{J}=\mathcal{O}(H)$ as algebras.

Pull $J$ back to $\mathcal{O}(G)$ along the obvious Hopf algebra surjective map $\mathcal{O}(G)\twoheadrightarrow \mathcal{O}(H)$ determined by $Y,W\mapsto 0$, $X\mapsto X$, and $V\mapsto V$. By (\ref{nm}) and (\ref{ersq}), it is straightforward to find out that ${}_J\mathcal{O}(G)_{J}$ is generated as an algebra by $X,Y,V,W$, subject to the relations
$$
[X,Y]=[X,V]=[Y,V]=[Y,W]=0,\,\,[W,X]=Y,\,\,[W,V]=Y^2/2.
$$
In particular $X,V$ span a lie algebra isomorphic to $\mathfrak{h}$, $W,Y$ span an abelian lie algebra $\mathfrak{a}$, and we have algebra isomorphisms 
$${}_J\mathcal{O}(G)_{J}\cong U(\mathfrak{a})\rtimes U(\mathfrak{h})\cong \mathcal{O}(G/H)\#{}_J\mathcal{O}(H)_{J} ,$$ where $[X,W]=-Y$, $[X,Y]=0$, $[V,Y]=0$ and $[V,W]=-Y^2/2$. (See Theorem \ref{noethdomunipgen}.)

Take $g:=g(x_0,v_0,w_0,0)\in G$. Then $H_{g}=H$, $\mathcal{I}(Z_{g})=(Y, W-w_0)$, and ${}_J\mathcal{O}(Z_{g})_J\cong \mathbb{C}[X,V]\cong U(\mathfrak{h}_g)=U(\mathfrak{h})$ as algebras.

Take $g:=g(0,0,0,y_0)$, $y_0\ne 0$. Then $g^{-1}=g(0,0,0,-y_0)$. We have 
$$H_{g}=\{1+xE_{12}-\frac{y_0}{2}xE_{13}\mid x\in\mathbb{C}\}\cong \mathbb{G}_a,$$
and
$$HgH=\left\{
\left( \begin{array}{cccc}
1 & x & v& w\\
0 & 1 & y_0 & \frac{y_0^2}{2} \\
0 & 0 & 1 & y_0\\ 
0 & 0 & 0 & 1
\end{array} \right)\middle | x,v,w\in \mathbb{C}\right\}.$$
It follows that $\mathcal{I}(Z_{g})=(Y-y_0)$, and $[W,V]\equiv y_0^2/2\neq 0$ on $HgH$. Thus, we have 
$${}_J\mathcal{O}(Z_{g})_J\cong \mathbb{C}[X]\ot \mathbb{C}[V][W;d/dV]\cong U(\mathfrak{h}_g)\ot \mathscr{W}_1$$ as algebras. 

Finally, note that we have $p=\frac{y_0}{2} X+V$, $\iota_L^*(\rho_g(W)-W-W(g))=y_0p$, $p$ vanishes on $H_g$, and $R^J(p,\iota^*(X))=1$. Also, for $l=l(x,v)\in L$, we have $\tau(l)=1-y_0(y_0x/2+v)E_{14}\in C$ (see \ref{I}).
\end{example}

\begin{example}\label{heisen3}
Let $G$ and $\g$ be as in Example \ref{heisen1}. 
Set 
$$r:=a\wedge c + d\wedge b.$$ Then $r$ is a non-degenerate solution of the CYBE in $\wedge^2\mathfrak{g}$, corresponding to the symplectic structure $\omega$ on $\mathfrak{g}$ determined by $\omega(a,c)=\omega(d,b)=1$. 
Let $J:=1+r/2+\cdots$ be the corresponding minimal Hopf $2$-cocycle for $G$ \footnote{It is known that $J$ has this form (see, e.g., \cite{EG1}).}. It is straightforward to verify that $$J(X,V)=J^{-1}(V,X)=J(W,Y)=J^{-1}(Y,W)=1/2,$$ $$J(V,X)=J^{-1}(X,V)=J(Y,W)=J^{-1}(W,Y)=-1/2,$$ and $J,J^{-1}$ vanish on other pairs of generators. 

By (\ref{nm}), it is straightforward to find out that the minimal cotriangular Hopf algebra ${}_J\mathcal{O}(G)_{J}$ is generated as an algebra by $X,Y,V,W$, such that
$$
[W,X]=Y,\;\;[W,V]=Y^2/2 +X,
$$
and other pairs of generators commute.  
Set $X':=Y^2/2 +X$. Then $X',Y,V,W$ span a Lie algebra of dimension $4$ such that $[W,V]=X'$ and $[W,X']=Y$, hence isomorphic to $\mathfrak{g}$. Thus, ${}_J\mathcal{O}(G)_{J}\cong U(\mathfrak{g})$ as algebras (see Theorem \ref{noethdomminunip}). 
\end{example}

\begin{example}\label{heisen4}
Let $G:=U_4$ be the $6$-dimensional unipotent algebraic group of $4$ by $4$ upper triangular matrices over $\mathbb{C}$.
Its coordinate Hopf algebra $\mathcal{O}(G)=\mathbb{C}[F_{12},F_{23},F_{34},F_{13},F_{24},F_{14}]$ is a polynomial algebra, with $F_{12},F_{23},F_{34}$ being primitive, 
$$\Delta(F_{13})=F_{13}\ot 1 + 1\ot F_{13} +F_{12}\ot F_{23},\,\,\,\Delta(F_{24})=F_{24}\ot 1 + 1\ot F_{24} +F_{23}\ot F_{34}$$
and 
$$\Delta(F_{14})=F_{14}\ot 1 + 1\ot F_{14} +F_{13}\ot F_{34} + F_{12}\ot F_{24}.$$ 

Let $H:=\{1+xE_{12}+uE_{34}|x,u\in \mathbb{C}\}\subset G$. Then $H\cong\mathbb{G}_a^2$ is a closed subgroup of $G$, and $\mathcal{O}(H)=\mathbb{C}[X,U]$ is a polynomial Hopf algebra. Let $a:=\frac{\partial}{\partial X}$, $c:=\frac{\partial}{\partial U}$, $r:=a\wedge c$, and $J:=e^{r/2}$.

Pull $J$ back to $\mathcal{O}(G)$ along the Hopf algebra surjective homomorphism $\mathcal{O}(G)\twoheadrightarrow \mathcal{O}(H)$ determined by $F_{23},F_{13},F_{24},F_{14}\mapsto 0$, $F_{12}\mapsto X$, and $F_{34}\mapsto U$. Then it is straightforward to verify that $_J{}\mathcal{O}(G/H)$ is generated as an algebra by 
$F_{23},F_{13},Y,V$, where $Y:=F_{24}-F_{23}F_{34}$ and $V:=F_{14}-F_{13}F_{34}$, such that 
$$[Y,F_{13}]=F_{23}^2,\,\,\,[F_{13},V]=F_{23}F_{13},\,\,\,[Y,V]=F_{23}Y,$$
and other pairs of generators commute. 
In particular, $_J{}\mathcal{O}(G/H)$ is not commutative (see Remark \ref{typically}).
\end{example}

\begin{example}\label{heisen5}
Let $G:=U_4$ be as in Example \ref{heisen4}. Let  
$$H= \left \{
\left( \begin{array}{cccc}
1 & x & v & w \\
0 & 1 & y & \frac{y^2}{2} \\
0 & 0 & 1 & y\\ 
0 & 0 & 0 & 1
\end{array} \right)
\middle | x, v, w, y\in \mathbb{C}
\right \},
$$
and $J$ be as in Example \ref{heisen3} (where $H$ is denoted there by $G$).

Pull $J$ back to $\mathcal{O}(G)$ along the surjective Hopf algebra homomorphism $\mathcal{O}(G)\twoheadrightarrow \mathcal{O}(H)$ determined by 
$$F_{12}\mapsto X,\,\,F_{13}\mapsto V,\,\,F_{14}\mapsto W,\,\,F_{23},F_{34}\mapsto Y,\,\,F_{24}\mapsto Y^2/2.$$ Then using (\ref{ersq}), it is straightforward to verify that ${}_J\mathcal{O}(G)_{J}$ is generated as an algebra by 
$\{F_{ij}\}$, such that  
$$[F_{14},F_{12}]=F_{34},\,\,[F_{13},F_{14}]=F_{24}-F_{12}-F_{23}F_{34},\,\,[F_{14},F_{24}]=F_{23}-F_{34},$$
and other pairs of generators commute.

Let $C:=\{1+wE_{14}\mid w\in \mathbb{C}\}$. Then $C\subset H$ is central in $G$ (see \ref{II}). 
We have
$$L= 
\left\{\left(\begin{array}{cccc}
1 & x & v& 0\\
0 & 1 & 0 & 0 \\
0 & 0 & 1 & 0\\ 
0 & 0 & 0 & 1
\end{array}\right) \middle | x,v\in \mathbb{C}
\right \}\subset \overline{H}= 
\left\{\left(\begin{array}{cccc}
1 & x & v& 0\\
0 & 1 & y & \frac{y^2}{2} \\
0 & 0 & 1 & y\\ 
0 & 0 & 0 & 1
\end{array} \right)\middle | x, v, y\in \mathbb{C}
\right \}$$
and 
$$A=\left\{\left(\begin{array}{cccc}
1 & 0 & 0& 0\\
0 & 1 & y & \frac{y^2}{2} \\
0 & 0 & 1 & y\\ 
0 & 0 & 0 & 1
\end{array} \right)\middle | y\in \mathbb{C}
\right \}\cong \mathbb{G}_a.$$

Now take $g:=1+E_{34}$. Then we have 
$$H_{g}=\{1+xE_{12}+vE_{13}+wE_{14}\mid x,v,w\in\mathbb{C}\}\cong \mathbb{G}_a^3,\,\,\dim(HgH)=5,\,\,d=1,$$
$$\overline{H}_{\bar g}=\{1+xE_{12}+vE_{13}\mid x,v\in\mathbb{C}\}\cong \mathbb{G}_a^2,\,\,\dim(\overline{H}\bar g\overline{H})=4,$$
$$L=L_{{\bar g}}\cong \mathbb{G}_a^2,\,\,\dim(L\bar g L)=2,\,\,\bar d=0,$$
$$\bar g A=\left\{\left(\begin{array}{cccc}
1 & 0 & 0& 0\\
0 & 1 & y & \frac{y^2}{2} \\
0 & 0 & 1 & y+1\\ 
0 & 0 & 0 & 1
\end{array} \right)\middle | y\in \mathbb{C}
\right \},\,\,A\bar g =\left\{\left(\begin{array}{cccc}
1 & 0 & 0& 0\\
0 & 1 & y & y+\frac{y^2}{2} \\
0 & 0 & 1 & y+1\\ 
0 & 0 & 0 & 1
\end{array} \right)\middle | y\in \mathbb{C}
\right \}$$
and
$$A\bar g A=\left\{\left(\begin{array}{cccc}
1 & 0 & 0& 0\\
0 & 1 & y & z+\frac{y^2}{2}\\
0 & 0 & 1 & y +1\\ 
0 & 0 & 0 & 1
\end{array} \right)\middle | y,z\in \mathbb{C}
\right \}.
$$
In particular, $A$ is not $\bar g$-invariant (see \ref{II.2}).

Now it is straightforward to verify that
$$HgH=\left\{
\left( \begin{array}{cccc}
1 & x & v& w\\
0 & 1 & y & z \\
0 & 0 & 1 & y+1\\ 
0 & 0 & 0 & 1
\end{array} \right)\middle | x,v,w,y,z\in \mathbb{C}\right\}.$$
Thus, 
$\mathcal{I}(HgH)=\langle F_{34}-F_{23}-1\rangle$, and we see that $F_{14},F_{24}$ generate a Weyl subalgebra in ${}_J\mathcal{O}(Z_g)_{J}$. (In the notation of \ref{II.2}, we have $W=F_{14}$, $V=F_{24}$, $X=F_{23}$, and $Y=F_{34}$.)

Finally, let $A,B,C,T,Q$ be the images of $F_{12},F_{13},F_{34},F_{24},F_{14}$ in ${}_J\mathcal{O}(Z_g)_{J}$, respectively. Then ${}_J\mathcal{O}(Z_g)_{J}$ is generated as an algebra by 
$A,B,C,T,Q$, such that 
$[A,Q]=-C$, $[B,Q]=T-A-C(C-1)$, $[T,Q]=1$, 
and other pairs of generators commute. Thus, replacing $A$ with $-A$ and setting $P:=T-C(C-1)$, we see that  
we have an algebra isomorphism
$${}_J\mathcal{O}(Z_g)_{J}\cong\mathbb{C}[A,B,C,P][Q;\delta],\,\, 
\delta:=C\frac{\partial}{\partial A} + (P+A)\frac{\partial}{\partial B} + \frac{\partial}{\partial P}.$$ Set $A':=A-CP$ and $B':=B-(P+A)P+(C+1)P^2/2$, and let $\mathscr{W}$ denote the Weyl subalgebra generated by $P,Q$. Then we have an algebra isomorphism 
$${}_J\mathcal{O}(Z_g)_{J}\cong\mathbb{C}[A',B',C]\ot \mathscr{W}\cong U^{\omega_g}(\mathfrak{h}_g)\ot \mathscr{W}.$$ 
(See Question \ref{problem2}.)
\end{example}

\section{The Hopf algebra ${}_J\mathcal{O}(G)_{J}$ for connected nilpotent $G$}\label{nilphopf}

In this section, we let $G=T\times U$ be a connected nilpotent algebraic group over $\mathbb{C}$, where $T$ is the maximal torus of $G$ and $U$ is the unipotent radical of $G$. Let $A:=\mathcal{O}(G)$. Let $\mathcal{O}(U)=\mathbb{C}[y_1,\dots,y_m]$ be as in Section \ref{algstrhopf}, and let $A_0:=\mathcal{O}(T)=\mathbb{C}[X(T)]=\mathbb{C}[x_1^{\pm 1},\dots,x_k^{\pm 1}]$. Recall that $A=\mathcal{O}(T)\ot \mathcal{O}(U)$ as Hopf algebras. Finally, set $A_i:=A_0[y_1,\dots,y_i]$, $1\le i\le m$ (so, $A_m=A$).

\begin{lemma}\label{lemmanilpot}  
Let $J$ be a Hopf $2$-cocycle for $A$, and 
let $\cdot$ denote the multiplication in ${}_J A_{J}$. The following hold:
\begin{enumerate}
\item
We have a Hopf filtration on ${}_J A_{J}$:
\begin{equation*}\label{filt2}
A_0\subset {}_J (A_1)_{J}\subset\cdots \subset {}_J (A_i)_{J}\subset\cdots\subset {}_J (A_m)_{J}={}_J A_{J}.
\end{equation*}
\item
For every $i$, the Hopf algebra ${}_J (A_i)_{J}$ is generated by $y_i$ over ${}_J (A_{i-1})_{J}$.
\item
For every $j<i$, we have
\begin{eqnarray*}
\lefteqn{[y_i,y_j]:=y_i\cdot y_j - y_j\cdot y_i}\\ 
& = & \sum Y_i'Q(Y_i'',y_j) + \sum Y_j'Q(y_i,Y_j'')+\sum Y_i'Y_j'Q(Y_i'',Y_j'')\\ 
& + & \sum \left(J^{-1}(Y_i',Y_j')J(Y_{i2}'',Y_{j2}'')-J_{21}^{-1}(Y_i',Y_j')J_{21}(Y_{i2}'',Y_{j2}'')\right)Y_{i1}''Y_{j1}''.
\end{eqnarray*}
Hence, $[y_i,y_j]$ belongs to $A_{i-1}^+$. 
\item
$y_1,y_2$ are central primitives in ${}_J A_{J}$.
\item
For every $i,j$, we have 
\begin{eqnarray*}
\lefteqn{[y_i,x_j]=x_j\sum Q(Y_i'',x_j)Y_i'}\\
& + & x_j\sum \left(J^{-1}(Y_i',x_j)J(Y_{i2}'',x_j)-
J_{21}^{-1}(Y_i',x_j)J_{21}(Y_{i2}'',x_j)\right)Y_{i1}''.
\end{eqnarray*}
\item
The linear map
$\delta_i:{}_J (A_{i-1})_{J} \to {}_J (A_{i-1})_{J},\,\,s\mapsto [y_i,s]$, is an algebra derivation of ${}_J (A_{i-1})_{J}$ for every $i$.
\item
For every $i$, ${}_J (A_i)_{J}\cong {}_J (A_{i-1})_{J}[y_i;\delta_i]$ as Hopf algebras.
\end{enumerate}
\end{lemma}

\begin{proof}
(1)--(4) are similar to Lemma \ref{lemmaunip}. As for (5), we have 
\begin{eqnarray*}
\lefteqn{[y_i,x_j]=x_j\left(J^{-1}(y_j,x_j)-J_{21}^{-1}(y_j,x_j)+Q(y_j,x_j)\right)}\\ 
& + & x_j\left(\sum Q(Y_i'',x_j)Y_i'+\sum J^{-1}(Y_i',x_j)Y_{i1}''J(Y_{i2}'',x_j)\right)\\ 
& - & x_j\left(\sum J_{21}^{-1}(Y_i',x_j)Y_{i1}''J_{21},(Y_{i2}'',x_j)\right),
\end{eqnarray*}
and since $\epsilon([y_i,x_j])=0$, $J^{-1}(y_j,x_j)-J_{21}^{-1}(y_j,x_j)+Q(y_j,x_j)=0$. 

Finally, (6) follows from (3) and (5), and (7) from (2) and (6).
\end{proof}

For every $i,j$, define $p_{ij}\in \mathcal{O}(U)^+$ as follows:
\begin{eqnarray*}
\lefteqn{p_{ij}:=\sum Q(Y_i'',x_j)Y_i'}\\
& + & \sum\left(J^{-1}(Y_i',x_j)J(Y_{i2}'',x_j)-
J_{21}^{-1}(Y_i',x_j)J_{21}(Y_{i2}'',x_j)\right)Y_{i1}''.
\end{eqnarray*}

\begin{theorem}\label{noethdomnilp}
The following hold:
\begin{enumerate}
\item
$\mathcal{O}(T)$ and ${}_J \mathcal{O}(U)_{J}$ are Hopf subalgebras of ${}_J\mathcal{O}(G)_{J}$.
\item
The group $X(T)$ acts on ${}_J \mathcal{O}(U)_{J}$ by automorphisms via $$x_j^{-1}y_ix_j=y_i+p_{ij}.$$
\item
${}_J\mathcal{O}(G)_{J}\cong {}_J \mathcal{O}(U)_{J} \rtimes \mathbb{C}[X(T)]$ is a smash product algebra.
\item
We have ${\rm Rep}({}_J\mathcal{O}(G)_{J})={\rm Rep}({}_J \mathcal{O}(U)_{J})^{X(T)}$.
\item
The Hopf algebra ${}_J\mathcal{O}(G)_{J}$ is an affine Noetherian domain with Gelfand-Kirillov dimension $\dim(G)$.
\end{enumerate}
\end{theorem}

\begin{proof}
(1) and (2) follow from Lemma \ref{lemmanilpot}, (3) follows from (2), and 
(4)--(5) follow from (3) and Corollary \ref{noethdomunipgr}.
\end{proof}

\begin{remark}
Theorems \ref{repthm}, \ref{noethdomnilp}(4) imply a classification of finite dimensional irreducible representations of ${}_J\mathcal{O}(G)_{J}$.
\end{remark}

\begin{example}
Let $U$ be the Heisenberg group as in Example \ref{heisen} (except, there it is denoted by $G$). 
Let $G:=\mathbb{G}_m\times U$. Then $G$ is a connected (non-unipotent) nilpotent algebraic group over $\mathbb{C}$, and we have $\mathcal{O}(G)=\mathbb{C}[F^{\pm 1},X,Y,V]$, where $F$ is a grouplike element and $X,Y,V$ are as in Example \ref{heisen}. The Lie algebra $\g$ of $G$ has basis $f,a,b,c$, where $f:=F\frac{\partial}{\partial F}$, and $a,b,c$ are as in Example \ref{heisen}. By \cite[Theorem 5.3 \& Proposition 5.4]{G}, the classical $r$-matrix $r:=f\wedge (a+b)$ for $\g$ corresponds to a Hopf $2$-cocycle $J$ for $G$. It is straightforward to verify that ${}_J\mathcal{O}(G)_{J}$ is generated as an algebra by $F,X,Y,V$, such that $[V,F]=F(Y-X)$, or equivalently, $F^{-1}VF=V+Y-X$, and other pairs of generators commute. Thus, we have ${}_J\mathcal{O}(G)_{J}\cong \mathbb{C}[X,Y,V]\rtimes \mathbb{C}[F^{\pm }]$ as algebras, with nontrivial action.
\end{example}

\end{document}